\documentclass[reqno,11pt]{amsart}

\usepackage{palatino}

\usepackage{amsmath}
\usepackage{amsfonts,amssymb,amsmath,latexsym,wasysym,mathrsfs,bbm,stmaryrd}
\usepackage[all]{xy}
\usepackage{color}
\usepackage{epsfig}
\usepackage{graphics,graphicx,moreverb,hangcaption}
\def\Te{\mathbb R}

\def\Z{\mathbb Z}
\def\N{\mathbb N}

\def\1{\mathbbm 1}

\def\PP{\mathscr{P}}

\newcommand\dd{\mathrm{{\bf d}}}
\newcommand\dT{\mathrm{{\bf d}}_T}
\newcommand\D{\mathcal{D}}
\newcommand\Dr{\mathcal{D}_r}
\newcommand\Cr{\mathcal{C}_r}
\newcommand\T{\mathcal{T}}
\newcommand\Tr{\mathcal{T}_r}
\newcommand\pT{p_{\: T, \sigma}}
\newcommand\I{\mathrm{{\bf I}}}
\newcommand{\n}{\mathrm{{\bf n}}}
\newcommand{\m}{\mathrm{{\bf m}}}
\newcommand{\LT}{\mathrm{LT}}
\newcommand{\CF}{\mathrm{CF}}
\newcommand{\Rot}{\mathrm{Rot}}
\newcommand{\Ell}{\mbox{\boldmath$\ell$}}

\newcommand{\interior}[1]{\raise0.2ex\hbox{$\displaystyle{\mathop{#1}^{\circ}}$}}
\newcommand{\mx}[1]{\mathbf{#1}}

\renewcommand\phi{\varphi}
\renewcommand\emptyset{\varnothing}

\newtheorem{theorem}{Theorem}[section]

\newtheorem{proposition}[theorem]{Proposition}
\newtheorem{corollary}[theorem]{Corollary}
\newtheorem{conjecture}[theorem]{Conjecture}
\newtheorem{lemma}[theorem]{Lemma}

\numberwithin{equation}{section}

\theoremstyle{definition}
\newtheorem{defn}[theorem]{Definition}
\newtheorem{definition}[theorem]{Definition}

\theoremstyle{remark}
\newtheorem*{remark*}{Remark}
\newtheorem{remark}[theorem]{Remark}

\newtheorem*{example*}{Example}

\begin{document}

\title{Enumeration of Non-Crossing Pairings on Bit Strings}
\author{Todd Kemp$^{(1)}$}
\address{$(1)$ Department of Mathematics, MIT \\ 77 Massachusetts Avenue, Cambridge, MA \; 02139}
\email{tkemp@math.mit.edu}
\author{Karl Mahlburg$^{(2)}$}
\address{$(2)$ Department of Mathematics, MIT \\ 77 Massachusetts Avenue, Cambridge, MA \; 02139}
\email{mahlburg@math.mit.edu}
\author{Amarpreet Rattan$^{(3)}$}
\address{$(3)$ Heilbronn Institute, University of Bristol \\ Bristol, UK}
\email{amarpreet.rattan@bristol.ac.uk}
\author{Clifford Smyth$^{(4)}$}
\address{$(4)$ Mathematics and Statistics Department, University of North Carolina Greensboro \\ 116 Petty Building, Greensboro, NC \; 27402 }
\email{cdsmyth@uncg.edu}

\begin{abstract}
A non-crossing pairing on a bitstring matches $1$s and $0$s in a manner such that the pairing diagram is nonintersecting. By considering such pairings on arbitrary bitstrings $1^{n_1} 0^{m_1} \dots 1^{n_r} 0^{m_r}$, we generalize classical problems from the theory of Catalan structures.  In particular, it is very difficult to find useful explicit formulas for the enumeration function $\phi(n_1, m_1, \dots, n_r, m_r)$, which counts the number of pairings as a function of the underlying bitstring.  We determine explicit formulas for $\phi$, and also prove general upper bounds in terms of Fuss-Catalan numbers by relating non-crossing pairings to other generalized Catalan structures (that are in some sense more natural).  This enumeration problem arises in the theory of random matrices and free probability.

\end{abstract}

\maketitle

\section{Introduction}

\subsection{The Knights and Ladies of the Round Table}

The objective of this paper is to address an enumeration problem which can be described in the following medieval terms.  King Arthur wishes to host a soir\'ee for the Knights of the Round Table.  There are $n$ Knights, so Arthur invites $n$ Ladies to Camelot for the event.  His intent is to seat each Knight next to a Lady so they may fraternize; however, before he has the chance to set the seating arrangement, all $2n$ guests seat themselves at the Round Table in a random configuration.  King Arthur now must consider two questions.  (1) Is it possible for the Knights and Ladies to pair off to talk while seated in this manner at the Round Table, {\em with no two conversations crossing}?  (2) If so, in how many {\em distinct} ways may they pair off to converse with no conversations crossing?

\medskip

Let $0$ denote a seat occupied by a Knight, and let $1$ denote a seat occupied by a Lady.  Figure \ref{fig Knights and Ladies} shows a possible random seating configuration (with $n=12$), and two possible solutions to King Arthur's problem.

\begin{figure}[htbp]
\begin{center}
	\includegraphics{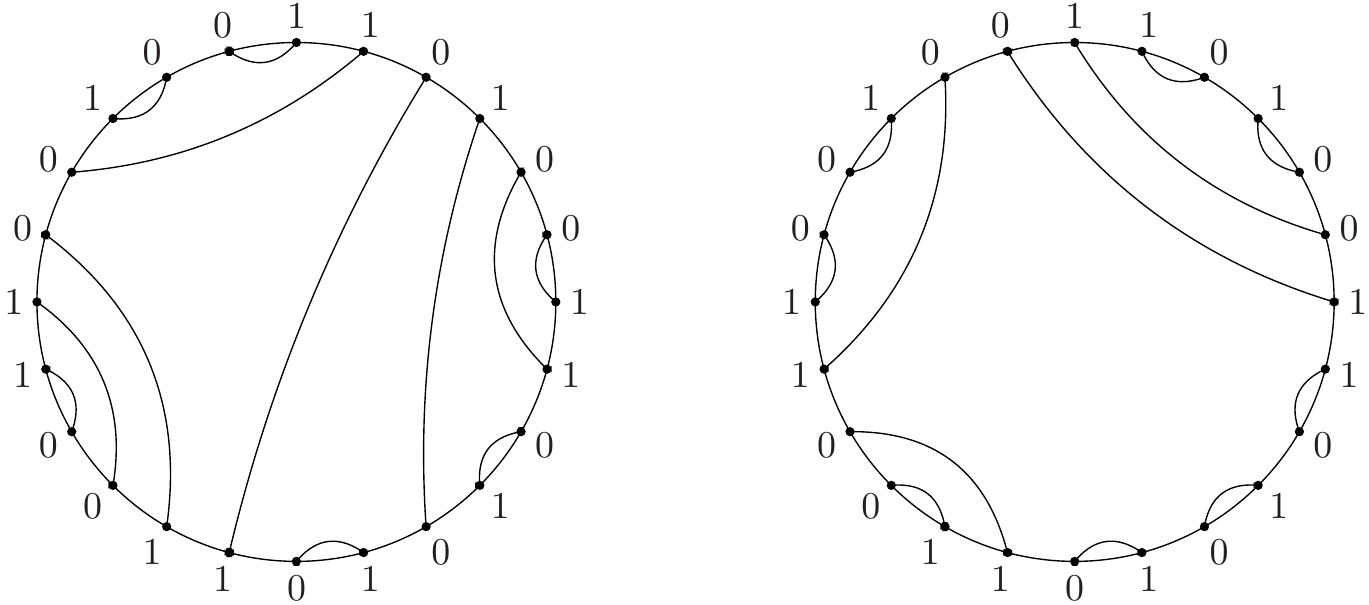}
	\caption{\label{fig Knights and Ladies}A random seating arrangement of $12$ Knights ($0$) and $12$ Ladies ($1$), along with two possible non-crossing conversation patterns.}
\end{center}
\end{figure}

It is relatively straightforward to see that there always exists a solution to the
Problem of the Knights and Ladies of the Round Table for a given ordering of knights
and ladies.  In short, in any random seating arrangement of the $2n$ guests, there must be a Knight sitting next to a Lady; pair them off and remove them from consideration, since their conversation cannot cross any other couple's.  This reduces the problem to one with $2(n-1)$ guests.  Continuing by induction, the problem is reduced to the case $n=1$ (a single couple) for which there is obviously a solution.  It can happen that this procedure produces the {\em unique} non-crossing pairing, if it happens that all the Knights are seated together as are all the Ladies; in string notation, the configuration is $11\cdots 100\cdots 0 = 1^n 0^n$.  As we will see in Section \ref{section rough bounds}, typically the number is much larger than $1$.

\begin{remark} The only string for which there is a unique non-crossing pairing is $1^n 0^n$.  To be precise, this is the only one up to rotation: as we are thinking of the configuration on a circle, rather than on a line as written here, we identify $1^n 0^n$ with $1^k 0^n 1^{n-k}$ for $0\le k\le n$.  We will discuss this inherent rotational symmetry in Section \ref{section rotation}. \end{remark}

The set of such non-crossing pairings of a random bit string is important in less medieval applications in modern mathematics.  The motivation for the problem comes from Random Matrix Theory and Free Probability Theory.  Let $X_N$ denote an $N\times N$ matrix whose entries are all independent complex normal random variables: for $1\le j,k\le N$, $[X_N]_{jk} = a_{jk}+ib_{jk}$ where $\{a_{jk},b_{jk}\,;\,1\le j,k\le N\}$ are independent normal random variables with variance $\frac{1}{2N}$.  (Such $X_N$ is sometimes referred to as a {\em Ginibre Ensemble} $GinUE_N$.)  The Hermitan cousin $G_N$ to $X_N$ (namely $G_N = \frac12(X_N + X_N^\ast)$) is called a {\em GUE} or {\em Gaussian Unitary Ensemble}, and has been studied by physicists for over half a century.  In that case, the object of interest is the distribution of eigenvalues.  Since $X_N$ is not a normal matrix, however, it's eigenvalues carry less information about the matrix ensemble itself.  Instead one studies the {\em matrix moments} $\frac1N\mathrm{Tr}\,(X_N^{n_1} X_N^{\ast m_1}\cdots X_N^{n_r} X_N^{\ast m_r})$ (these carry the same information as the eigenvalues in the Hermitian-case; in general they contain vastly more data).  The connection between these moments and our interests is summed up in the following proposition, whose proof can be found in \cite{Nica Speicher Book}.

\begin{proposition} \label{prop matrix moments} Let $X_N$ be a random matrix with independent complex  normal entries (real and imaginary parts of variance $\frac{1}{2N}$).  Let $n_1,\ldots,n_r$ and $m_1,\ldots,m_r$ be non-negative integers.  Then the mixed matrix moment
\[ \mathrm{Tr}\left[ (X_N)^{n_1} (X_N^\ast)^{m_1} \cdots (X_N)^{n_r} (X_N^\ast)^{m_r} \right] \]
converges, as $N\to\infty$, to the number of non-crossing pairings in the problem of Knights and Ladies of the Round Table with a seating arrangement of $1^{n_1} 0^{m_1}\cdots 1^{n_r} 0^{m_r}$.
\end{proposition}

Therefore our goal, in a sense, is to calculate all asymptotic mixed matrix moments of a Ginibre Ensemble.  However, the set of such non-crossing pairings is far more generic than in this one example.  In \cite{Nica Speicher Fields Paper}, the authors introduced {\em $\mathscr{R}$-diagonal operators}, which represent limiting eigenvalue distributions of a large class of non-self-adjoint random matrices with {\em non}-independent entries (but that nevertheless have nice symmetry and invariance properties).  Such ensembles of random matrices have recently played very important roles in free probability and beyond: for example, in \cite{Haagerup}, Haagerup has produced the most significant progress towards the resolution of the Invariant Subspace Conjecture in decades, and his proof is concentrated in the theory of $\mathscr{R}$-diagonal operators.  In \cite{Kemp dilation}, the first author showed that the asymptotic mixed matrix moments of $\mathscr{R}$-diagonal random matrices are controlled, in an appropriate sense, by the set of non-crossing pairings we consider in this paper.  Indeed, the results of the present paper followed from discussions motivated by applications to $\mathscr{R}$-diagonal operators.

\medskip

\subsection{Main Theorems}

As will become apparent, finding a closed formula for the number of non-crossing pairings for
any random seating arrangement seems to be an unfeasible goal.  Given the motivation
above, we wish to study these numbers regardless, so we now set the notation to be used throughout this paper.

\begin{definition} \label{def NC2} A {\em pairing} of the set $\{1,\ldots,2n\}$ is a collection of $n$ pairs $\pi = \{\{i_1,j_1\},\ldots,\{i_n,j_n\}\}$ with the property that $\{i_1,i_2,\ldots,i_n,j_1,j_2,\ldots,j_n\} = \{1,\ldots,2n\}$.  A {\em crossing} of $\pi$ is a pair of pairs $\{i_1,i_2\},\{j_1,j_2\} \in \pi$ such that $i_1<j_1<i_2<j_2$; $\pi$ is called {\em non-crossing} if it has no crossings.  The set of non-crossing pairings of $\{1,\ldots,2n\}$ is denoted $NC_2(2n)$.  \end{definition}
\begin{figure}[htbp]
\begin{center}
\includegraphics{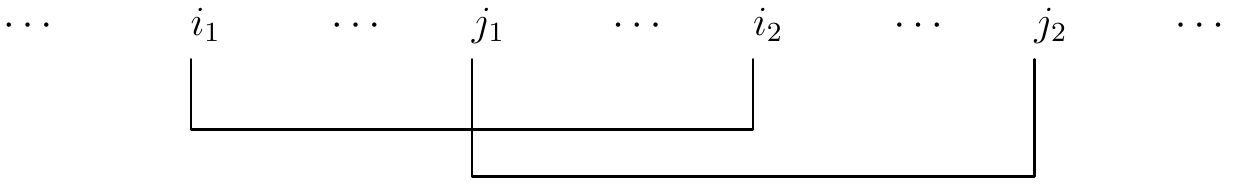}
\caption{A crossing in a partition.}
\label{figure crossing partition}
\end{center}
\end{figure}

\begin{definition} \label{def NC2S} Let $\mx{S}=s_1s_2\cdots s_{2n}$ be a bit string of length $2n$ (so $s_j\in\{0,1\}$ for $1\le j\le 2n$).  A pairing of $\{1,\ldots,2n\}$ is called a {\em non-crossing $\mx{S}$-pairing} if $\pi\in NC_2(2n)$ and if, for each pair $\{i,j\}\in\pi$, $s_i\ne s_j$.  (So $\pi$ pairs $1$s to $0$s.)  The set of non-crossing $\mx{S}$-pairings is denoted $NC_2(\mx{S})$. \end{definition}

This paper concerns the enumeration of $NC_2(\mx{S})$ for arbitrary bit strings $\mx{S}$.  For convenience, we denote for any bit string $\mx{S}$
\begin{equation} \label{eq phi} \phi(\mx{S}) =  |NC_2(\mx{S})|. \end{equation}
It is clear that a string $\mx{S}$ must have even length for $\phi(\mx{S})$ to be non-zero.  What's more, since any $\pi\in NC_2(\mx{S})$ must pair $1$s with $0$s, $NC_2(\mx{S}) = \emptyset$ unless $\mx{S}$ is {\em balanced}: it must have as many $1$s as $0$s.  Any string $\mx{S}$ may be written in the form
\begin{equation} \label{eq string} \mx{S} = 1^{n_1} 0^{m_1} \cdots 1^{n_r} 0^{m_r} \end{equation}
where $n_1$ or $m_r$ may be $0$ but all other exponents $n_j,m_j$ are strictly positive integers.  For reasons that will be made clear in Section \ref{section rotation}, we will only concern ourselves here with strings beginning with $1$ and ending with $0$, and so in (\ref{eq string}) we have all exponents $n_1,\ldots,m_r>0$.  We may therefore think of $\phi$ as a function
\[ \phi\colon \bigsqcup_{r=1}^\infty \N^{2r} \to \N\cup\{0\}, \]
where $\N = \{1,2,\ldots\}$, via the identification $\phi(n_1,m_1,\ldots,n_r,m_r) := \phi(1^{n_1} 0^{m_1}\cdots 1^{n_r} 0^{m_r})$.  We will use $\phi$ in this dual manner throughout in order to ease notational complexity.  It will also be convenient to extend the definition of $\phi$ to an arbitrary number of integer arguments; we set
\begin{align}
\label{equation:phiintegers}
\phi(n_1, m_1, \dots, n_r, m_r) & := 0 \qquad \text{if} \; n_i < 0 \; \text{or} \; m_i < 0 \; \text{for any} \; i,\\
\phi(n_1, m_1, \dots, n_i, 0, n_{i+1}, m_{i+1}, \dots, n_r, m_r) & := \phi(n_1, m_1, \dots, n_i + n_{i+1}, m_{i+1}, \dots, n_r, m_r), \notag \\
\phi(m_0, n_1, m_1, \dots, n_r, m_r) & := \phi(0, m_0, n_1, m_1, \dots, n_r, m_r). \notag
\end{align}

\medskip

The following proposition lists exact values for the enumeration function $\phi$ on
arguments of lengths $2$, $4$, and $6$.  Elementary proofs of Propositions \ref{prop
low-run formulas}.a and \ref{prop low-run formulas}.b are given in Sections
\ref{section rotation} and \ref{sec:recur}, respectively.

\begin{proposition} \label{prop low-run formulas}a) We have $\phi(n, n) =1$, and if $n_1+n_2=m_1+m_2$ then
\begin{equation*} \label{eq low-run 1} \phi(n_1,m_1,n_2,m_2) =
	1+\min\{n_1,m_1,n_2,m_2\}. \end{equation*}
b) For $n_1,n_2,n_3\in\N$, let $i = \min\{n_1,n_2,n_3\}$, and let $j = \min\left(\{n_1,n_2,n_3\}-\{i\}\right)$.  Then
\begin{equation*} \label{eq low-run 2} \phi(n_1,n_1,n_2,n_2,n_3,n_3) =
	\textstyle{\frac{1}{2}i^2 + ij + \frac{3}{2} i + j + 1}. \end{equation*}
 \end{proposition}
The increasing complexity of the formulas in Proposition \ref{prop low-run formulas}
make it seem that a general formula, even if it could be written down concisely, would have little use.  Indeed, it is possible to write down an explicit formula for $\phi(n_1,m_2,n_2,m_2,n_3,m_3)$ in general, rather than the symmetric case $n_j=m_j$ given above, but the formula takes more than two full lines.  In certain special cases, however, an explicit formula can be written down.

\begin{proposition} \label{prop Fuss-Catalan} The values of $\phi$ on regular strings $(1^n 0^n)^r$ (corresponding to the case $n_1 = \cdots = n_r = m_1 = \cdots = m_r$) is given by
\begin{equation} \label{eq Fuss-Catalan} \phi\left((1^n 0^n)^r\right) = C^{(n)}_r = \frac{1}{nr+1}\binom{(n+1)r}{r}. \end{equation}
\end{proposition}
The numbers appearing in (\ref{eq Fuss-Catalan}) are called {\em Fuss-Catalan numbers}.  In the special case $n=1$ (corresponding to the alternating string $1010\cdots 10$), they yield the Catalan numbers $C_r = \frac{1}{r+1}\binom{2r}{r}$.  These enumerate {\em all} non-crossing pairings $NC_2(2r)$; indeed, any non-crossing pairing is automatically a $1010\cdots10$-pairing, for if $i,j$ are both even or both odd and $(i,j)\in\pi$ then, since there are an {\em odd} number of points between $i$ and $j$, at least one pair in $\pi$ must have an end between $i$ and $j$ and the other outside, producing a crossing in $\pi$.  A more sophisticated version of this reasoning, together with the recurrence for the Fuss-Catalan numbers, forms a proof of Proposition \ref{prop Fuss-Catalan} as discussed in \cite{REU}.  The proposition was also proved in a more topological manner by the first author in \cite{Kemp Speicher}, which relies on the non-crossing partition multichain enumeration results in \cite{Edelman} (proofs are also essentially contained in \cite{Larsen} and \cite{Oravecz}.)  We will rely upon Proposition \ref{prop Fuss-Catalan} in much of the remainder of the introduction; indeed, the technology we develop here (and continue in  forthcoming papers) is all in the general scheme of {\em Fuss-Catalan structures}.

\medskip

The parameter $r$ plays an important role in all that follows, so we give it a name.

\begin{definition} Given a string $\mx{S} = 1^{n_1} 0^{m_1}\cdots 1^{n_r} 0^{m_r}$ with $n_j,m_j>0$, say that $\mx{S}$ has {\em $r$ runs}. \end{definition}
The strings that appear in Proposition \ref{prop Fuss-Catalan} are the most symmetric strings with a fixed number of runs.  The authors' general intuition, after many calculations and numerical experiments, was that these highly symmetric words should maximize $\phi$ among all strings with given length and number of runs.  This intuition is borne out in the results of Proposition \ref{prop low-run formulas}.  This simple-to-state conjecture turns out to be remarkably intricate, and is indeed our main theorem.

\begin{theorem} \label{Main Theorem} Let $\{n_1,\ldots,n_r\}$ and $\{m_1,\ldots,m_r\}$ be sets of positive integers with common sum $n$.  Then
\begin{equation} \label{eq main theorem} \phi(1^{n_1} 0^{m_1}\cdots 1^{n_r} 0^{m_r}) \le \phi((1^k 0^k)^r), \end{equation}
where $k = \lceil n/r \rceil$.  In particular, in the case that $r$ divides $n$, the string $(1^k 0^k)^r$ maximizes $\phi$ among all strings with length $2n$ and $r$ runs. \end{theorem}
\noindent The latter statement, describing the theorem in terms of maximizers of $\phi$, has a slightly stronger form that remains conjectural in the general case that the number of runs does not divide half the length.

\begin{conjecture} \label{conjecture} Let $n,r\in\N$ with $r\le n$.  Write $n = \ell r + a$ with $0\le a<r$.  Then for any bit string $\mx{S}$ with length $2n$ and $r$ runs,
\[ \phi(\mx{S}) \le \phi((1^{\ell+1} 0^{\ell+1})^a (1^\ell 0^\ell)^{r-a}). \]
\end{conjecture}
\noindent The string $(1^{\ell+1} 0^{\ell+1})^a (1^\ell 0^\ell)^{r-a}$ is one interpretation of the ``most symmetric string with length $2n$ and $r$ runs''.  If $a>0$ then $k= \lceil n/r \rceil = \ell+1$, and so the string $(1^k 0^k)^r$ appearing on the right-hand-side of (\ref{eq main theorem}) is $(1^{\ell+1} 0^{\ell+1})^r$, which is slightly longer than we believe to be necessary.  Let us note that we can prove Conjecture \ref{conjecture} in the case that $n_1\le n_2 \le \cdots \le n_r$ or $m_1\le m_2\le \cdots \le m_r$, and also in somewhat greater generality; this is discussed in Sections \ref{section tree ordering} and \ref{section Proof}.

\subsection{Outline} This paper is organized as follows.  Section \ref{section:Basic} describes some of the basic properties of noncrossing pairings, including certain symmetries, a recurrence relation, and elementary bounds for $\phi$.  Section \ref{section:phi*} then reduces the general problem of counting noncrossing pairings (thereby bounding $\phi$) to the simpler problem of counting pairings for a restricted class of ``locally symmetric'' bitstrings.  We obtain exact formulas for the restricted situation, although we are left short of our main goal of proving Theorem \ref{Main Theorem}; and the section concludes with several conjectural observations that would complete the proof and are also of independent interest.  We take a different approach to the problem in Section \ref{section:MainProof}, where we injectively map noncrossing pairings into labeled trees (which are ``generalized Catalan structures''), and use the combinatorics of these trees to finally prove Theorem \ref{Main Theorem}.  Finally, Section \ref{section:Conclusion} ends the paper with some concluding remarks about the connections between noncrossing pairings and other common combinatorial objects.

\bigskip

\section{Basic Properties of Non-Crossing Pairings on Bit Strings}
\label{section:Basic}

\subsection{Symmetries and Rotational Invariance} \label{section rotation}

The first elementary observation about non-crossing pairings on bit strings is that they display
both cyclic and reflective symmetry.  More precisely, the lattice of all non-crossing pairings on bitstrings of length $2n$ is invariant under cyclic permutations of $\{1,\ldots,2n\}$, as well as under the reflection permutation.  Indeed, these permutations generate the lattice automorphism group of the full lattice of non-crossing {\em partitions} (cf.\ \cite{Nica Speicher Book}).  Given a pairing in $NC_2(\mx{S})$ for some string $\mx{S}$ of length $2n$, such a rotation or reflection of the underlying ordered set $\{1,\ldots,2n\}$ typically does not respect the string $\mx{S}$.  However, if the string is similarly permuted, then the correspondence is clear; it also makes no difference in the pairings if we interchange 1s and 0s.

\begin{definition}
If $\mx{S} = s_1 s_2 \cdots s_{2n}$, the {\em reflection} of $\mx{S}$ is $\mathrm{Refl}(\mx{S}) := s_{2n} \cdots s_2 s_1$, for $1 \leq k \leq 2n,$ the {\em rotation} of $\mx{S}$ by $k$ is $\mathrm{Rot}_k(\mx{S}) := s_k s_{k+1} \cdots s_{2n} s_1 s_2 \cdots s_{k-1}.$  The {\em negation} of $\mx{S}$ replaces each $s_i$ by $1 - s_i$, and is denoted by $1 - \mx{S}$.
\end{definition}

\begin{proposition} \label{prop rotation} For any $\mx{S}$ as above and any integer $k$, $\phi(\mx{S}) = \phi(\Rot_k(\mx{S})) = \phi\left(\mathrm{Refl}_k(\mx{S})\right) = \phi(1-\mx{S})$.
\end{proposition}

We omit the proof of this proposition, but refer the reader to Figure \ref{figure
rotation}.
\begin{figure}[htbp]
	\begin{center}
		\includegraphics{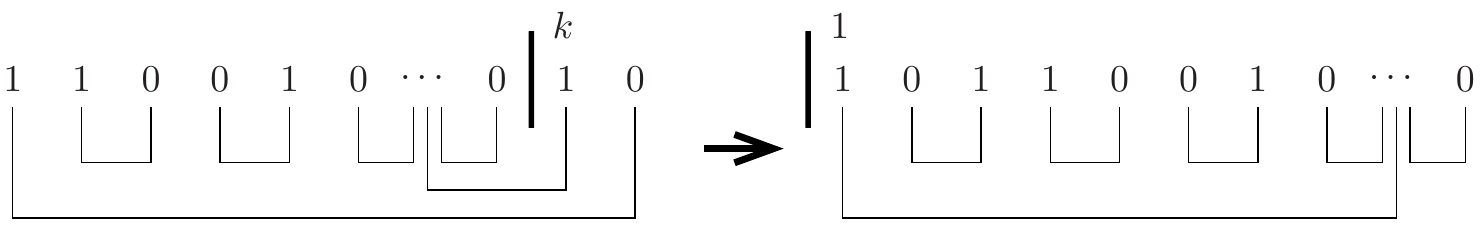}
		\caption{\label{figure rotation}The action of the rotation $\Rot_k$.}
	\end{center}
\end{figure}

While Proposition \ref{prop rotation} makes it clear that it is
	natural to draw non-crossing pairings of bit strings around a circle as in
	the motivating Figure \ref{fig Knights and Ladies}, it is often convenient
	to use the linear representation and we continue to do so while keeping Proposition
	\ref{prop rotation} in mind.
Rotations provide a key tool in the analysis that follows in Sections
\ref{section:phi*} and \ref{section:MainProof}.  Proposition \ref{prop rotation} will
be used to prove a recurrence relation for $\phi$ in Section \ref{sec:recur}, but we will also need the following useful symmetries of $\phi$ written as an arithmetic function.
\begin{corollary}
\label{corollary:phisymmetry} For any integers $n_i, m_i,$
\begin{align*}
\phi(n_1,m_1,\ldots,n_r,m_r) & = \phi(m_1,n_2,\ldots,m_r,n_1) \qquad \mathrm{(Rotation)}, \\
\phi(n_1,m_1,\ldots,n_r,m_r) & = \phi(m_r,n_r,\ldots,m_1,n_1) \qquad \mathrm{(Reflection)}.
\end{align*}
\end{corollary}
In particular, we may assume without loss of generality that $n_1 = \min\{n_1,m_1,\ldots,n_r,m_r\}$; in this case, we call $n_1$ {\em minimal}.  We will make this minimum-first assumption frequently in what follows.

\subsection{The Lattice Path Representation of a Bit String} \label{section preliminaries}

Let $\mx{S}$ be a balanced bit string, beginning with $1$ and ending in $0$: $\mx{S} = s_1s_2\cdots s_{2n} = 1^{n_1}0^{m_1}\cdots 1^{n_r}0^{m_r}$ where $n_1+\cdots+n_r = m_1+\cdots+m_r = n$ and $n_j,m_j>0$ for $1\le j\le r$.  To reiterate, we say the {\em length} of $\mx{S}$ is $2n$, and we say $\mx{S}$ has $r$ {\em runs}.  We frequently make use of a convenient representation of $\mx{S}$ as a lattice path.

\begin{definition} \label{def lattice path} Given a string $\mx{S}=s_1s_2\cdots
	s_{2n}$, set $Y_0 = 0,$ and $Y_i := \sum_{j=1}^i (-1)^{s_i + 1}$ for $1 \leq
	i \leq 2n$.  Set $p_i = (i,Y_i)\in\Te^2$.  Define $\PP(\mx{S})$ in $\Te^2$ to be the piecewise linear path consisting of the union of the $2n$ line segments $p_{i-1}p_i$ for $1 \le i \le 2n$  (i.e., 1s correspond to ``up'', and 0s to ``down'').  We refer to $\PP(\mx{S})$ as the {\em lattice path} of $\mx{S}$.
\end{definition}
\begin{definition}
\label{definition:height}
Given $\mx{S}=s_1s_2\cdots s_{2n}$ as above, set $m := \min\{Y_1, \dots, Y_{2n}\}.$  Then the {\em height} of $s_i$ for $1 \leq i \leq 2n$ is
\begin{equation*}
h_i := \begin{cases}
Y_i - m  \qquad & \text{if} \; s_i = 1, \\
Y_i - m + 1  \qquad & \text{if} \; s_i = 0.
\end{cases}
\end{equation*}
The {\em height} of the path $\PP(\mx{S})$ is the maximum value of any $h_i$, and is denoted $h(\mx{S})$.
\end{definition}
\begin{remark*}
The shift by $m$ in Definition \ref{definition:height} ensures that the lowest height is always $1$.
\end{remark*}

\begin{figure}[htbp]
\begin{center}
\includegraphics{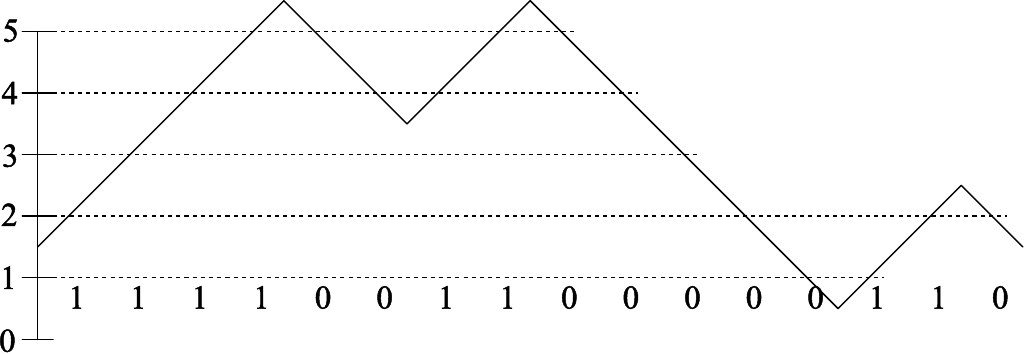}
\caption{\label{figure lattice path} The lattice path $\PP(\mx{S})$ of the string $\mx{S}=1^40^21^20^51^20$.  The heights are labeled starting at the minimum point, from $1$ to $5 = h(\mx{S})$.}
\end{center}
\end{figure}

The lattice path $\PP(\mx{S})$ is useful in understanding noncrossing pairings of $\mx{S}$ due to the following simple observation: if $\{i,j\}$ is in such a pairing $\pi$ (assume without loss of generality that $s_i = 1$ and $s_j=0$),  then there must be equal numbers of $1$s and $0$s among the $s_k$ with $i<k<j$; in other words, an equal number of up-slopes and down-slopes.  This proves that the pairings of $\mx{S}$ are restricted by heights.
\begin{lemma} \label{lemma:pairing height} If $\pi\in NC_2(\mx{S})$ and $\{i,j\}\in\pi$, then $h_i = h_j$.
\end{lemma}
\begin{corollary}
\label{corollary:peak reduction} For any bitstring $\mx{S}$ we have $\phi(\mx{S}) =
\phi(\widetilde{\mx{S}})$, where $\widetilde{\mx{S}}$ is the result of removing the tallest peak and lowest valley in $\mx{S}$ to level them with the second tallest and second lowest.
\end{corollary}

We summarize the information in $\PP(\mx{S})$ by writing the height $h_i$ of each slope $s_i$ as a label above the corresponding bit in the string.  The result for the string in Figure \ref{figure lattice path} is
\[ \mathop{1}^2\,\mathop{1}^3\,\mathop{1}^4\,\mathop{1}^5\,\mathop{0}^5\,\mathop{0}^4\,
\mathop{1}^4\,\mathop{1}^5\,\mathop{0}^5\,\mathop{0}^4\,\mathop{0}^3\,\mathop{0}^2\,
\mathop{0}^1\,\mathop{1}^1\,\mathop{1}^2\,\mathop{0}^2. \]
This example illustrates that the first label need not be $1$.

\medskip

The simple observations above allow us to enumerate pairings of small strings quite
easily.  Following is the proof of Proposition \ref{prop low-run formulas}.a.

\begin{proof}[Proof of Proposition \ref{prop low-run formulas}.a] For the first statement, the labels of $1^n 0^n$ are
\[ \mathop{1}^1\,\mathop{1}^2\, \cdots\,\mathop{1}^{n-1}\, \mathop{1}^n\,\mathop{0}^n\,\mathop{0}^{n-1}\,\cdots\,\mathop{0}^2\,\mathop{0}^1. \]
Note that each label from $1$ through $n$ appears exactly twice: once on a $1$ and
once on a $0$.  This means that there can be at most one pairing in $NC_2(1^n 0^n)$,
and it is simple to check that the requisite totally-nested pairing is non-crossing.
So $\phi(n, n) = 1$.  (This example demonstrates the content of Corollary \ref{corollary:peak reduction}; in the highest peak and lowest valley, the pairings are prescribed to be locally nested.)

\medskip

Now consider $1^{n_1} 0^{m_1} 1^{n_2} 0^{m_2}$.  Corollary
\ref{corollary:phisymmetry} allows us to assume that $i$, the minimum of the $n_j$, is $n_1$.  Then $m_1 = \mu_1+i$ where $\mu_1\ge 0$.  Also, $n_2 \ge \mu_1 + i$, for $i \le m_2 = n_1 + n_2 - m_1 = i + n_2 - (\mu_1+i)$, and so we may write $n_2 = \nu_2 + \mu_2 + i$ for $\nu_2\ge 0$, and subtracting we also have $m_2 = \nu_2 + i$.  We therefore write
\[ 1^{n_1} 0^{m_1} 1^{n_2} 0^{m_2} = 1^i 0^i 0^{\mu_1} 1^{\mu_1} 1^i 1^{\nu_2} 0^{\nu_2} 0^i, \]
and the corresponding lattice path is represented in Figure \ref{figure lattice path 2}.  The height labels are as follows (to save space, we have subtracted $\mu_1$ from all labels):
\[ \mathop{1}^{1}\, \cdots \, \mathop{1}^{i}\, \mathop{0}^{i} \, \cdots \, \mathop{0}^{1}\,\boxed{\mathop{0}^{0}\,\cdots \mathop{0}^{1-\mu_1}\,\mathop{1}^{1-\mu_1}\,\cdots\,\mathop{1}^{0}}\,\mathop{1}^{1}\,\cdots\,\mathop{1}^{i}\,\boxed{\mathop{1}^{i+1}\,\cdots\,\mathop{1}^{i+\nu_1}\, \mathop{0}^{i+\nu_1}\,\cdots\,\mathop{0}^{i+1}}\,\mathop{0}^{i}\,\cdots\,\mathop{0}^{1}.  \]
The boxed regions contain bits with unique labels: there is only one $1$ (and one $0$) for each (shifted) label $0$ through $1-\mu_1$ and $i+1$ through $i+\nu_1$, and so those intervals must be paired in nested fashion as above.   The remaining unpaired bits form the reduced string $(1^i0^i)^2$, and from Proposition \ref{prop Fuss-Catalan}, $\phi((1^i 0^i)^2) = C^{(2)}_i = 1+i$, as required.
\begin{figure}[htbp]
\begin{center}
\includegraphics[width=0.8\textwidth]{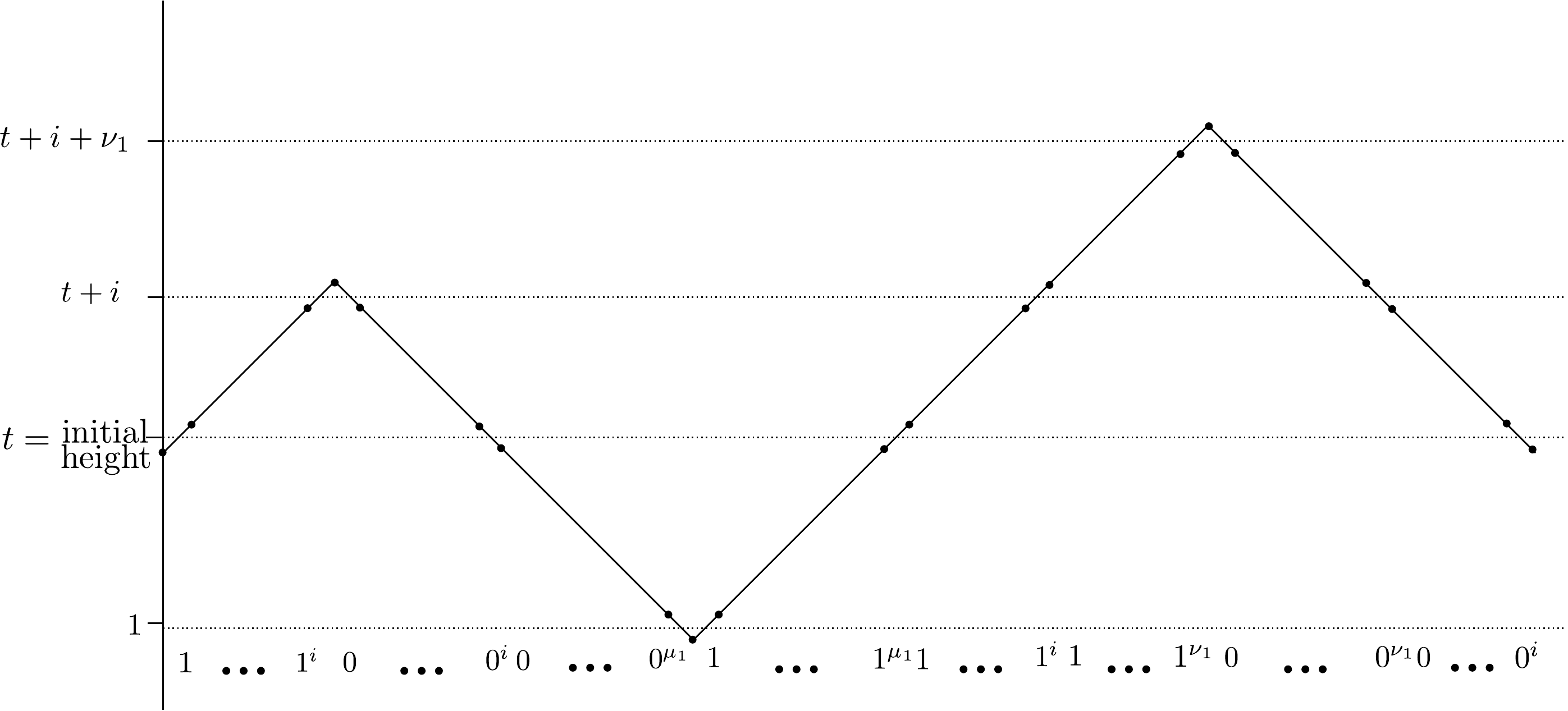}
\caption{\label{figure lattice path 2}The lattice path corresponding to the string
$1^{n_1} 0^{m_1} 1^{n_2} 0^{m_2} = 1^i 0^i 0^{\mu_1} 1^{\mu_1} 1^i 1^{\nu_2}
0^{\nu_2} 0^i$.  Here we have represented the path with run heights in a particular ranking (first minimal as always, then second and fourth, then third); of course, any ordering is possible, but the above proof works in general.}
\end{center}
\end{figure}
\end{proof}

\subsection{Recurrence and Functional Equations}\label{sec:recur}

The following multivariate recurrence relation for $\phi$ underlies many of the arguments in Section \ref{section:phi*}.

\begin{theorem}\label{theorem recurrence} Let $n_1,\ldots,n_r,m_1,\ldots,m_r>0$ with $n_1+\cdots+n_r = m_1+\cdots +m_r$.  Define, for $1\le k\le r$,
\begin{equation*}
d_k := -(n_1 + \cdots + n_k) + (m_1 + \cdots + m_k).
\end{equation*}
Then we have the recurrence
\begin{equation*}
\phi(n_1,m_1,\ldots,n_r,m_r) \\
 = \sum_{k=1}^r \phi(n_1-1,m_1,\ldots,n_k,m_k-d_k-1)\cdot \phi(d_k, n_{k+1},m_{k+1},\ldots,n_r,m_r).
\end{equation*}
\end{theorem}
\begin{remark} Using the conventions of (\ref{equation:phiintegers}), the $k$-th term vanishes whenever $d_k < 0$ or $m_k - d_k \leq 0.$  Furthermore, the second factor in the last term of the right-hand side is just $\phi(\emptyset) = 1$, and the first factor is $\phi(n_1-1,m_1,\ldots,n_r,m_r-1)$ (which is always non-zero since $d_r = 0$ for balanced strings).  We choose not to write this as a separate term in order to keep the recurrence relation more concise.  \end{remark}

\begin{proof}

The recurrence arises by considering the possible pairings of the first $1$.  If we write $\mx{S} = 1^{n_1}0^{m_1}\cdots 1^{n_r}0^{m_r}$, we trivially have
\begin{equation} \label{eqn proof recurrence 1} \phi(\mx{S}) = |NC_2(\mx{S})| = \sum_{k=1}^r |\{\pi\in NC_2(\mx{S})\,:\,\pi\text{ pairs the first} \; 1 \; \text{to a} \; 0 \; \text{in run} \; k\}|. \end{equation}

If the height condition in Lemma \ref{lemma:pairing height} is not met by any of the $0$s in the $k$-th run, then there are no such pairings, and the $k$-th term in (\ref{eqn proof recurrence 1}) vanishes.  We use the numbers $d_k$ to measure the relevant heights (for convenience, we shift all heights so that $h_1 = 1$): the quantity $-d_k + 1$ is the height of the $k$-th valley, and$-d_k + m_k + 1$ is the height of the $k$-th peak.  If $d_k<0$, then the $k$-th run of $0$s lies entirely above height $1$, and if $-d_k+m_k >0$, then the $k$-th run of $0$s lies entirely below height $1$.  In either case, the $k$-th term in (\ref{eqn proof recurrence 1}) is zero.

Otherwise, there is a unique $0$ in the $k$-th run at the same height at the first $1$ in the first run.  Suppose that $s_1 = 1$ pairs to $s_p = 0$ in the $k$-th run of $0$s.  Once this pairing is made, the non-crossing condition on $\pi$ breaks the remaining bits into two disjoint substrings:
\[ 1^{n_1-1} 0^{m_1}\cdots 1^{n_k} 0^{m_k-d_k-1} \quad \text{ and } \quad 0^{d_k} 1^{n_{k+1}} 0^{m_{k+1}} \cdots 1^{n_r} 0^{m_r}. \]
The total number of pairings is then the product of the pairings on each substring.

\medskip

Thus in all cases, the $k$-th term in (\ref{eqn proof recurrence 1}) is
\begin{equation} \label{eqn proof recurrence 2}
\phi(1^{n_1-1} 0^{m_1}\cdots 1^{n_k} 0^{m_k-d_k-1})\cdot\phi(0^{d_k} 1^{n_{k+1}} 0^{m_{k+1}} \cdots 1^{n_r} 0^{m_r}). \end{equation}

\begin{figure}[htbp]
\begin{center}
\includegraphics{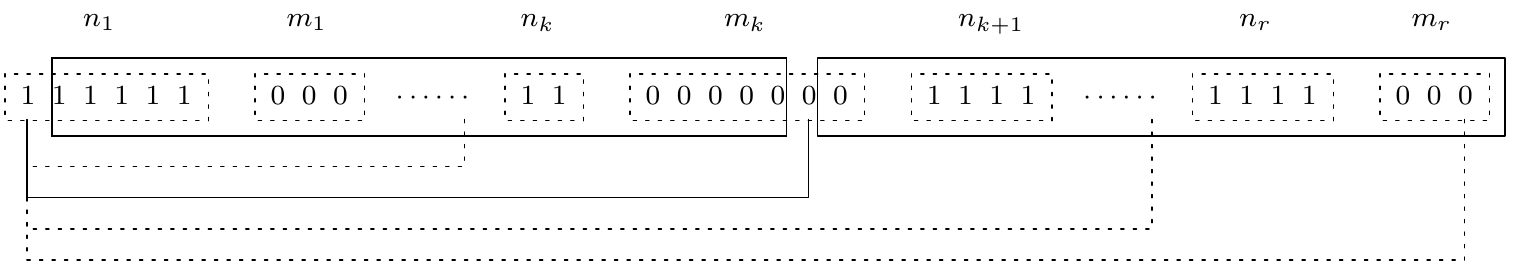}
\caption{\label{figure recurrence} Once a choice has been made for the pairing of the left-most $1$, the remaining allowed pairings are forced to decompose into the two substrings due to the non-crossing condition.}
\end{center}
\end{figure}

\end{proof}

\begin{remark} \label{remark string recurrence} Theorem \ref{theorem recurrence} is a precise quantification of the statement that for any string $\mx{S}$ beginning with $1$,
\[ \phi(\mx{S}) = \sum_{\mx{S} = 1\mx{R}0\mx{T}} \phi(\mx{R})\phi(\mx{T}). \]
The sum may be taken over all balanced substrings $\mx{R},\mx{T}$ of $\mx{S}$, although the requirement that the strings be balanced is in fact extraneous, since $\phi(\mx{R}) = 0$ automatically whenever $\mx{R}$ is not balanced.  \end{remark}

As an application of Theorem \ref{theorem recurrence}, we now complete the proof of Proposition \ref{prop low-run formulas}.

\begin{proof}[Proof of Proposition \ref{prop low-run formulas}.b] We wish to calculate $\phi(n_1,n_1,n_2,n_2,n_3,n_3)$; let us assume that $n_1 = \min\{n_1,n_2,n_3\} \equiv i$.  Then the two highest peaks are $n_2,n_3$, and following Corollary \ref{corollary:peak reduction} we then have $\phi(n_1,n_1,n_2,n_2,n_3,n_3) = \phi(i,i,j,j,j,j)$ where $j = \min\{n_2,n_3\}$.  The benefit of having $n_k = m_k$ for all $k$ as in this example is that each $d_k$ is $0$, and so Theorem \ref{theorem recurrence} gives directly
\[ \phi(i,i,j,j,j,j) = \phi(i-1,i-1)\phi(j,j,j,j) + \phi(i-1,i,j,j-1)\phi(j,j) + \phi(i-1,i,j,j,j,j-1). \]
From the proof of Proposition \ref{prop low-run formulas}.a, we know that $\phi(i-1,i-1) = \phi(j,j)=1$, while $\phi(j,j,j,j) = 1+j$ and $\phi(i-1,j,j,j-1) = 1+\min\{i-1,j,j-1\} = 1+i-1 = i$, and so
\begin{equation} \label{eqn 3-run recurrence} \phi(i,i,j,j,j,j) = 1+j +i + \phi(i-1,i,j,j,j,j-1). \end{equation}


For the remaining $\phi$ term on the right, it is convenient to rotate the string into one that begins and ends with a $0$, and interchange $1$s and $0$s:
\[ \mx{1}-\Rot_{i+j-1}(1^{i-1} 0^i 1^j 0^j 1^j 0^{j-1}) = 10^j1^j0^j1^{j-1}0^{i-1}1^{i-1}, \]
and therefore $\phi(i-1,j,j,j,j,j-1) = \phi(10^j1^j0^j1^{j-1}0^{i-1}1^{i-1})$.  Labeling this string we have
\[ \mathop{1}^j \mathop{0}^j \mathop{0}^{j-1} \cdots \mathop{0}^1 \mathop{1}^1 \cdots \mathop{1}^j
\mathop{0}^j \cdots \mathop{0}^1 \mathop{1}^1\cdots \mathop{1}^{j-1} \mathop{0}^{j-1} \cdots \mathop{0}^{j-i+1} \mathop{1}^{j-i+1} \cdots \mathop{1}^{j}. \]
There are only two $0$s with the highest label $j$ -- one in the first run and one in the second run.  Applying Theorem \ref{theorem recurrence}, we get
\[ \phi(10^j1^j0^j1^{j-1}0^{i-1}1^{i-1}) = \phi(\emptyset)\phi(0^{j-1}1^j0^j1^{j-1}0^{i-1}1^{i-1})
+ \phi(0^j1^j)\phi(0^{j-1}1^{j-1}0^{i-1}1^{i-1}). \]
As calculated above, $\phi(\emptyset) = 1$, $\phi(0^{j}1^j) = 1$, and
$\phi(0^{j-1}1^{j-1}0^{i-1}1^{i-1}) = 1+\min\{i-1,j-1\} = 1+i-1 = i$.  Finally, since
the string $0^{j-1}1^j0^j1^{j-1}0^{i-1}1^{i-1}$ has a unique tallest peak, following
Corollary \ref{corollary:peak reduction} we can reduce it to the same height as the second highest peak:
\[ \phi(0^{j-1}1^j0^j1^{j-1}0^{i-1}1^{i-1}) = \phi(0^{j-1}1^{j-1} 0^{j-1} 1^{j-1} 0^{i-1} 1^{i-1}) = \phi(i-1,i-1,j-1,j-1,j-1,j-1). \]
Let $\phi_{i,j} = \phi(i,i,j,j,j,j)$; then combining these calculations with (\ref{eqn 3-run recurrence}), we have
\begin{equation} \label{eqn 3-run recurrence 2} \phi_{i,j} = 1+j+2i + \phi_{i-1,j-1}. \end{equation}
Iterating (\ref{eqn 3-run recurrence}) $i$ times yields
\[ \phi_{i,j} = \sum_{k=0}^{i-1} [1+(j-k)+2(i-k)] + \phi_{0,j-i}, \]
and $\phi_{0,j-i} = \phi(0,0,j-i,j-i,j-i,j-i) = \phi(j-i,j-i,j-i,j-i) = 1+j-i$.
Summing all the parts yields $\phi_{i,j} = \frac{1}{2}i^2 + ij + \frac{3}{2}i + j +1$, as required.
\end{proof}

\begin{remark} Note that in the case $n_1 = \min\{n_1,n_2,n_3\} = \min\{n_2,n_3\}$,
	Proposition \ref{prop low-run formulas} gives
\[ \phi(n_1,n_1,n_2,n_2,n_3,n_3) = \frac{3}{2}n_1^2 + \frac{5}{2}n_1 + 1, \]
and it is easy to check that this is indeed equal to the Fuss-Catalan number $C^{(n_1)}_{3}$.  As this holds in particular when $n_1=n_2=n_3$, this reproves Proposition \ref{prop Fuss-Catalan} in the case $r=3$. \end{remark}

Theorem \ref{theorem recurrence} can be written as a functional equation for the generating function of $\phi$.  The following is stated in Example 16.17 in \cite{Nica Speicher Book}, where it is proved by very different means.

\begin{proposition} \label{prop generating function} Let $F(x_0,x_1)$ be the non-commutative formal power series generating function for $\phi$,
\[ F(x_0,x_1) = \sum_{n=0}^\infty \sum_{\mx{S}\in \{0,1\}^{n}} \phi(\mx{S})\, \mx{x}^\mx{S}, \]
where for any (not necessarily balanced) string $\mx{S}$ of length $n$, $\mx{x}^\mx{S}$ denotes the non-commutative monomial $x_{s_1} x_{s_2} \cdots x_{s_n}$.  Then $F=F(x_0,x_1)$ satisfies the non-commutative quadratic equation
\begin{equation*} F= 1 + x_0 F x_1 F + x_1 F x_0 F. \end{equation*}
\end{proposition}

\begin{proof} Set $G(x_0,x_1) = 1+x_0\,F(x_0,x_1)\,x_1\,F(x_0,x_1) + x_1\,F(x_0,x_1)\,x_0\,F(x_0,x_1)$; to be more precise,
\begin{equation} \label{eqn G} G \equiv 1+ \sum_{\mx{R},\mx{T}} \phi(\mx{R})\phi(\mx{S})\,(x_0\,\mx{R}\,x_1\,\mx{T}+x_1\,\mx{R}\,x_0\,\mx{T}). \end{equation}
$G$ is a formal power-series in $x_0,x_1$; denote its coefficient function as $\psi$, so $G(x_0,x_1) = \sum_\mx{S} \psi(\mx{S})\,\mx{x}^\mx{S}$.  Our goal is to show that $\psi = \phi$.  For a given string $\mx{S}$, (\ref{eqn G}) states that either $\mx{S}=\emptyset$ (in which case $\psi(\emptyset)=1=\phi(\emptyset)$), or
\[ \psi(\mx{S}) = \sum_{\mx{R},\mx{T}\atop\mx{S}=1\mx{R}0\mx{T}} \phi(\mx{R})\,\phi(\mx{T})
+  \sum_{\mx{R},\mx{T}\atop\mx{S}=0\mx{R}1\mx{T}} \phi(\mx{R})\,\phi(\mx{T}) \]
Of course, $\mx{S}$ either begins with $1$ or begins with $0$, so only one of the two sums above is non-zero.  We treat the case $\mx{S}$ begins with $1$.  Now, $\phi(\mx{R})=0$ whenever $\mx{R}$ is not balanced, and so we really have
\[ \psi(\mx{S}) = \sum_{\mx{R},\mx{T}\text{ balanced}\atop\mx{S}=1\mx{R}0\mx{T}} \phi(\mx{R})\,\phi(\mx{T}). \]
Remark \ref{remark string recurrence} explains that the above summation is actually equal to the summation on the right-hand-side of Theorem \ref{theorem recurrence}; thence, $\psi(\mx{S}) = \phi(\mx{S})$.  The case that $\mx{S}$ begins with $0$ is identical.
\end{proof}

\begin{remark} If the second term of the quadratic equation in Proposition \ref{prop
	generating function} is removed, what remains is identical to the standard
	recurrence satisfied by the generating function for {\em Dyck paths}
	\cite[Example 6.2.6]{Stanley}.
\end{remark}



\subsection{Rough Bounds} \label{section rough bounds} Our goal in this paper is to prove the sharp upper bound of Theorem \ref{Main Theorem}.  We end this section by providing a number of rougher bounds, both upper and lower, for $\phi$ on arbitrary balanced strings.

\begin{proposition} \label{prop NC2 lower bound} Let $\mx{S} = 1^{n_1}0^{m_1}\cdots1^{n_r}0^{m_r}$ be a balanced string, and let $i$ be the minimum block size, $i=\min\{n_1,m_1,\ldots,n_r,m_r\} \ge 1$.  Then
\[ \phi(\mx{S}) \ge (1+i)^{r-1}. \]
\end{proposition}

\begin{proof} The cases $r=1,2$ are proved in Proposition \ref{prop low-run formulas}, providing the base case for an induction.  If $r\ge 2$, let $\mx{S}=1^{n_1}0^{m_1}\cdots1^{n_r}0^{m_r}$ be a balanced string with $r$ runs and minimum block size $i$; without loss of generality we assume that $n_1$ is minimal ($n_1 = i$). Since both $m_1, m_r\ge i$, for any $0\le \ell \le i = n_1$ we may pair the last $\ell$ $1$s in the first block $1^{n_1}$ to the first $\ell$ $0$s, with the remaining $i-\ell$ $1$s pairing to the final $i-\ell \le m_r$ $0$s in the final block.  The remaining internal string is then $0^{m_1-\ell}1^{n_2}\cdots0^{m_{r-1}}1^{n_r}0^{m_r-(i-\ell)}$, which can be rotated to
\begin{equation*}
\mx{\tilde{S}} = 1^{n_2}0^{m_2}\cdots1^{n_r}0^{m_1+m_r-i}.
\end{equation*}
This is a balanced string with $r-1$ runs, and its minimum run length $\tilde{i} = \min\{n_2, m_2, \cdots, n_r, m_r + m_1 - i\} \geq i$.  The inductive hypothesis then implies that $\phi(\mx{\tilde{S}}) \geq (1+\tilde{i})^{r-2}$.
\begin{figure}[htbp]
\begin{center}
\includegraphics{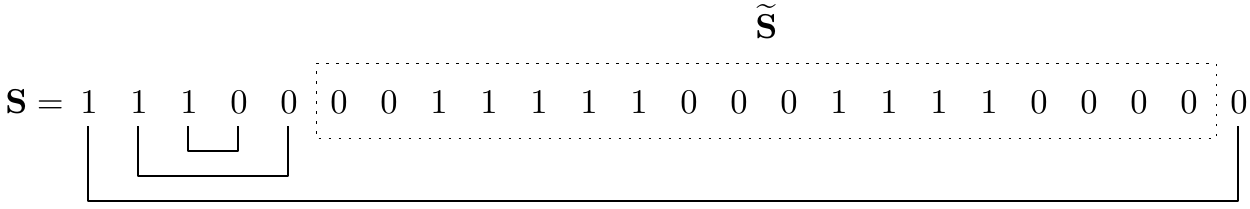}
\caption{One of the $i+1$ configurations for the first block of $1$s, yielding all the pairings of $\tilde{\mx{S}}$; in this example, $i=3$, and $\ell=2$.}
\label{fig NC_2 lower bound proof}
\end{center}
\end{figure}

Overall, for each choice of $0 \leq \ell \leq i$, we therefore have at least $(1+i)^{r-2}$ distinct pairings of $\mx{S}$, and the pairings for different $\ell$ are distinct.  This implies that $\phi(\mx{S}) \geq (1+i)^{r-1}$ as claimed.
\end{proof}

The preceding inductive proof actually yields a somewhat larger lower bound as follows.  Let
$i_1, \dots, i_{r-1}$ be the minima defined by the inductive process in the proof of
Proposition \ref{prop NC2 lower bound} (i.e. $i_1 = i$ is the global minimum in the
proof and each $i_{k+1}$ is the minimum of the leftover string after the inductive step has been
applied at stage $k$, so $i_2 = \tilde{i}$ from the proof, and so on).  The following
is a strengthening of Proposition \ref{prop NC2 lower bound}.
\begin{proposition}\label{prop:slightlystronger}
	Let $\mx{S}$ be defined as in Proposition \ref{prop NC2 lower bound} and
	$i_1, \dots, i_{r-1}$ be defined as in the preceding paragraph.   Then
\begin{equation*}
	\phi(\mx{S}) \geq (1 + i_{1}) \cdots (1 + i_{r-1}).
\end{equation*}
This bound is sharp, as demonstrated by the family of examples
\begin{equation*}
\mx{S} = 1^{a_1 + a_2} 0^{a_2} 1^{a_2 + a_3} 0^{a_3} \dots 1^{a_{r-1} + a_r} 0^{a_r} 1^{a_r + a_{r+1}} 0^{a_1 + a_2 + \dots + a_r + a_{r+1}},
\end{equation*}
where the $a_i$ are any positive integers.
\end{proposition}
\noindent The proof is similar to the proof of Proposition \ref{prop NC2 lower bound} and the
second claim follows from an application of Lemma \ref{lemma:pairing height}.


\medskip
In the other direction, we prove a simple upper bound (which is not sharp in general).
\begin{proposition} \label{prop upper bound on NC_2} Let $\mx{S}$ be a balanced string with lattice path height $h = h(\mx{S})$ and $r$ runs.  Then
\begin{equation} \label{eqn Fuss-Catalan upper bound} \phi(\mx{S}) \le C^{(h)}_r \le \frac{r^{r-1}}{r!}\,(1+h)^{r-1}. \end{equation}
\end{proposition}

\begin{proof} The proof relies on the following simple injection of pairings on $\mx{S}$ to pairings on $\mx{T} = (1^h0^h)^r$.  In $\mx{S}$, a run $1^{n_k}$ has heights $a, a + 1, \dots, a + n_k - 1$ where $h_i = a$ and all heights are in the range $[1, h]$.  The $k$-th run of $1$s in $\mx{T}$ hits every height $1, \dots, h$, and thus we use the {\em height-preserving map} from $\mx{S}$ to $\mx{T}$ (the situation for runs of $0$s is identical).  Furthermore, we preserve the pairings of $\mx{S}$ when injecting into $\mx{T}$.  If a run $1^{n_1}$ in $S$ ends at position $i$ with $h_i = a$, then the following run of $0$s also begins at the same height $h_{i+1} = a$.  This leaves excess bits $1^{h-a} 0^{h-a}$ in $\mx{T}$ at heights $a + 1, \dots, h$, which we pair locally.

\medskip

This gives the inclusion, and the first inequality then follows from Proposition \ref{prop Fuss-Catalan}.  The second inequality is an elementary rough estimate of the Fuss-Catalan number, which is left to the reader.  Figure \ref{fig proof of lemma upper bound} demonstrates the inclusion.  \end{proof}
\begin{figure}[htbp]
\begin{center}
\includegraphics{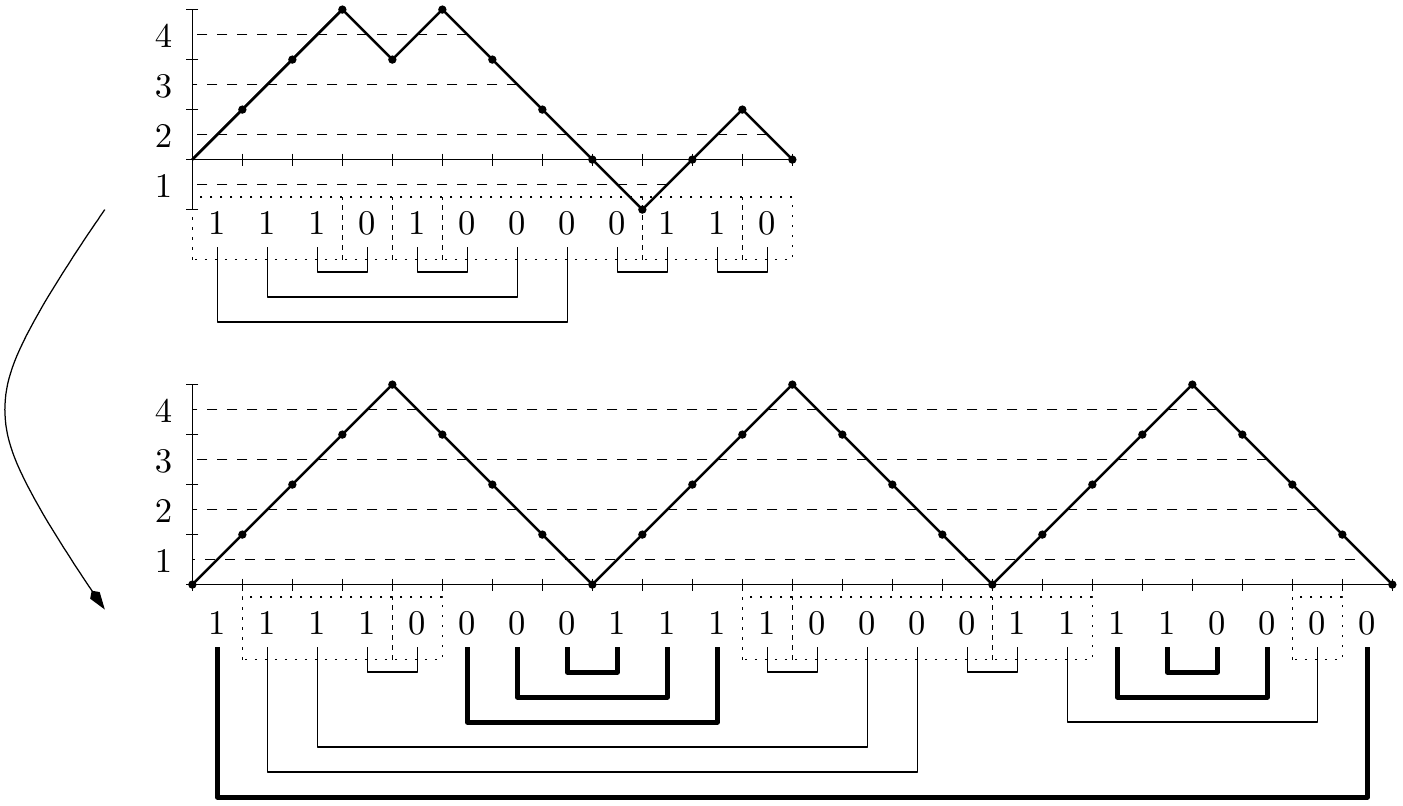}
\caption{ $\mx{S}$ is injected into $(1^h0^h)^r$, with extraneous labels (dark lines) paired locally.}
\label{fig proof of lemma upper bound}
\end{center}
\end{figure}

Note that the lattice path height is the smallest $h$ that can be used in the proof of Proposition \ref{prop upper bound on NC_2}, since all labels appearing in $\mathscr{P}(\mx{S})$ must be present in $\mathscr{P}\left((1^h0^h)^r\right)$.  Unfortunately, $h(\mx{S})$ can be quite large in comparison to the average (or even maximum) block size in $\mx{S}$: consider the string $(1^{k}0)^{\ell}(10^{k})^{\ell}$.  The maximum block size is $k$, while the lattice path height is $(k-1)\ell + 1$.  Indeed, this string has length $2(k+1)\ell$, and the height is nearly half the total length.  In general, a string of length $2n$ with $r$ runs can have height $n-r+1$, so the following corollary is essentially the best that can be said using height considerations.

\begin{corollary} \label{cor upper bound NC_2} Let $\mx{S}$ be a balanced string of
	length $2n$ ($n$ $1$s and $n$ $0$s), with $r$ runs.  Then
\begin{equation}\label{eqn upper bound NC_2} \phi(\mx{S}) \le \frac{r^{r-1}}{r!}\, (1+n)^{r-1}. \end{equation}
\end{corollary}

\begin{remark} The bound in Corollary \ref{cor upper bound NC_2} is quite large and never actually achieved, but it has the correct asymptotic behaviour in $r$ for fixed $n$.  Theorem \ref{Main Theorem} essentially states that $n$ may be replaced with $n/r$ in this corollary. \end{remark}

\section{Expansions for $\phi^*$}
\label{section:phi*}

As seen in Proposition \ref{prop low-run formulas}, exact formulas for $\phi(n_1, m_1, \dots, n_r, m_r)$ are significantly simpler when each $n_i = m_i$.  In this section we show that the general problem of finding upper bounds for $\phi$ can be replaced by the easier problem of finding upper bounds for the {\it symmetrized pairing function} $\phi^*$, which is defined as
\begin{equation}
\phi^*(n_1, n_2, \dots, n_r) := \phi(n_1, n_1, n_2, n_2, \dots, n_r, n_r).
\end{equation}
As before, we also set $\phi^*(\emptyset) = 1$ for technical reasons.  We also introduce a tree structure that leads to an exact formula for $\phi^*$, although some of the most important properties of the formula remain conjectural.

\subsection{Reduction to $\phi^\ast$}

Recall the recursion formula for $\phi$ from Theorem \ref{theorem recurrence},
\begin{align*}
\phi&(n_1, m_1, \dots, n_r, m_r) \\
& = \sum_{k=1}^r \phi(n_1 - 1, m_1, \dots, n_k, m_k - d_k - 1) \cdot
\phi(d_k, n_{k+1}, m_{k+1}, \dots, n_r, m_r),
\end{align*}
where $d_k = -(n_1 + \dots + n_r) + (m_1 + \dots + m_r).$  Using the rotational
symmetries of $\phi$ from Corollary \ref{corollary:phisymmetry}, we may as usual
assume that $n_1$ is minimal among $\{n_1, m_1, \dots, n_r, m_r\}$.  The following proposition combines these tools into another useful symmetry for $\phi$.

\begin{proposition}
\label{proposition:aphiIdentity}
If $n_1$ is minimal, then
\begin{equation*}
\phi(n_1 - a, n_1 - 1, n_2, n_2, \dots, n_r, n_r - (a-1)) = \phi(n_1 -
a, n_1, n_2, n_2, \dots, n_r, n_r - a).
\end{equation*}
\end{proposition}
\begin{remark}
Our proof of this identity is strictly algebraic.  It is an open problem to find a combinatorial proof that directly relates the two sets of noncrossing pairings.
\end{remark}
\begin{proof}
The proof is by induction on $n_1 - a.$  For the base case, suppose that $n_1 - a =
0.$   The left-hand side of the equality is then (using (\ref{equation:phiintegers}))
\begin{equation*}
\phi(0, n_1 - 1, n_2, n_2, \dots, n_r, n_r - (n_1 - 1)) = \phi^*(n_2,
\dots, n_r),
\end{equation*}
and the right-hand side is similarly
\begin{equation*}
\phi(0, n_1, n_2, n_2, \dots, n_r, n_r - n_1) = \phi^*(n_2, \dots,
n_r).
\end{equation*}

Now suppose that $n_1 - a \geq 1.$  Note that the values of $d_k$ when Theorem
\ref{theorem recurrence} is applied to the left hand side of the proposition statement are
particularly simple, as $d_k = -(a-1)$ for all $k<r$.  Thus we have the expansion
\begin{align}
\phi&(n_1 - a, n_1 - 1, n_2, n_2, \dots, n_r, n_r - (a-1)) \\
& = \sum_{k=1}^r \phi(n_1-(a+1), n_1 - 1, n_2, n_2, \dots, n_k, n_k -
a) \notag \\
& \qquad \qquad \times \phi(a-1, n_{k+1}, n_{k+1}, \dots, n_r, n_r - (a-1)) \notag \\
& = \sum_{k=1}^r \phi(n_1-(a+1), n_1 - 1, n_2, n_2, \dots, n_k, n_k -
a) \cdot \phi^*(n_{k+1}, \dots, n_r), \notag
\end{align}
where the second equality again uses (\ref{equation:phiintegers}).  We now apply the inductive hypothesis to the first terms in the summands and obtain
\begin{align}
\sum_{k=1}^r & \phi(n_1-(a+1), n_1, n_2, n_2, \dots, n_k, n_k -
(a+1)) \cdot \phi^*(n_{k+1}, \dots, n_r) \\
& = \phi(n_1-a, n_1, n_2, n_2, \dots, n_r, n_r - a), \notag
\end{align}
where we have applied Theorem \ref{theorem recurrence} in reverse to evaluate the sum.  Furthermore, there are no summands that unexpectedly vanish, since the condition that $n_1$ is minimal guarantees that $n_i - a \geq 0$ for all $i.$
\end{proof}

The $a=0$ case of this equality arises in the proof of the following recursive formula for $\phi^*$.
\begin{theorem}
\label{theorem:phi*Recurrence}
If $n_1$ is minimal, then
\begin{equation*}
\phi^*(n_1, n_2, \dots, n_r) = \sum_{i=1}^r \phi^*(n_1 - 1, n_2, \dots, n_i) \cdot
\phi^*(n_{i+1}, \dots, n_r).
\end{equation*}
\end{theorem}
\begin{proof}
When Theorem \ref{theorem recurrence} is applied to $\phi^*(n_1, \dots, n_r),$ we
have $d_k = 0$ for all $k<r$.
Therefore,
\begin{align*}
\phi^*(n_1, n_2, \dots, n_r) & = \sum_{i=1}^r \phi(n_1 - 1, n_1, n_2,
n_2, \dots, n_i, n_i-1)\cdot
\phi^*(n_{i+1}, \dots, n_r) \\
& = \sum_{i=1}^r \phi(n_1 - 1, n_1 - 1, n_2, n_2, \dots, n_i, n_i)\cdot
\phi^*(n_{i+1}, \dots, n_r),
\end{align*}
where the second equality utilizes Proposition \ref{proposition:aphiIdentity}.
\end{proof}

We now have the necessary tools to prove an important inequality between $\phi$ and $\phi^*$.
\begin{theorem}
\label{theorem:phiphi*}
If $n_1$ is minimal, then
\begin{equation*}
\phi(n_1, m_1, n_2, m_2, \dots, n_r, m_r) \leq \phi^*(n_1, n_2, \dots,
n_r).
\end{equation*}
\end{theorem}
\begin{proof}
The proof is by induction on $n_1 + n_2 + \dots + n_r$.  The only necessary base case
is $\phi(0,0) = 1 = \phi^*(0).$

Now suppose that $n_1 + \dots + n_r \geq 1.$  Theorem \ref{theorem recurrence} once again gives
\begin{align}
\label{eq:phiInequality1}
\phi(&n_1, m_1, n_2, m_2, \dots, n_r, m_r) \notag \\
& = \sum_{i=1}^r \phi(n_1 - 1, m_1, \dots, n_i, m_i + d_i - 1) \cdot
\phi(-d_i, n_{i+1}, m_{i+1}, \dots, n_r, m_r) \notag \\
& \leq \sum_{i=1}^r \phi(n_1 - 1, m_1, \dots, n_i, m_i + d_i - 1) \cdot
\phi(n_{i+1}, m_{i+1}, \dots, n_r, m_r - d_i) \notag
\end{align}
The inequality is due to (\ref{equation:phiintegers}), which implies that
\begin{equation*}
\phi(a, n_1, m_1, \dots, n_r, m_r) \leq \phi(n_1, m_1, \dots, n_r, m_r + a);
\end{equation*}
the left side is $0$ if $a$ is negative, and otherwise there is equality.  The induction hypothesis now implies that
\begin{align}
\phi(n_1, m_1, \dots, n_r, m_r) & \leq \sum_{i=1}^r \phi^*(n_1 - 1, n_2, \dots, n_i) \cdot
\phi^*(n_{i+1}, \dots, n_r) \notag \\
& = \phi^*(n_1, n_2, \dots, n_r), \notag
\end{align}
and the last equality uses Theorem \ref{theorem:phi*Recurrence}.
\end{proof}
\begin{remark*}
It is an open problem to prove this inequality combinatorially by finding an injection of noncrossing pairings.
\end{remark*}

The inequality in Theorem \ref{theorem:phiphi*} implies that upper bounds for $\phi^*$ also serve as upper bounds for $\phi.$  Therefore, the remainder of this section is devoted to achieving a better understanding of $\phi^*$.

\subsection{One-term recurrence for $\phi^\ast$}

Although we could directly use Theorem \ref{theorem:phi*Recurrence} to recursively compute values of $\phi^*(n_1, n_2, \dots, n_r)$, this would be very inefficient as it would require as many as $n_1 + n_2 + \dots + n_r$ recursive calls.  We can obtain a more useful formula that requires only $r$ recursive calls by ``unwinding'' the formula for $\phi^*$.

\begin{defn}
\label{defn:phi*S}
If $S = \{i_1, \dots, i_s\} \subset [1,r-1],$ with $i_1 \leq \dots \leq
i_s,$ then
\begin{equation*}
\phi_S^*(n_2, \dots, n_r) := \displaystyle \prod_{j = 1}^{s-1}
\phi^*(n_{i_j + 1}, \dots, n_{i_{j+1}})\cdot \phi^*(n_{i_s + 1}, \dots,
n_r).
\end{equation*}
\end{defn}

The subsets that we consider will be required to contain $1$, so for $S \subset [2,r],$ we adopt the notation $S_1 := S \cup \{1\}$.

\begin{theorem}
\label{theorem:phi*Recurrence1}
If $n_1$ is minimal, then $\phi^*(n_1) = 1, \phi^*(n_1, n_2) = 1 + n_1$ and for $r \geq 3,$
\begin{equation*}
\phi^*(n_1,n_2, \dots, n_r) = \sum_{S \subset [2,r-1]}
\binom{n_1+1}{|S_1|} \phi_{S_1}^*(n_2, \dots, n_r).
\end{equation*}
\end{theorem}

\begin{proof}
The cases $r=1$ and $r=2$ are special cases of Proposition \ref{prop low-run formulas}.  If $r \geq 3$ is fixed, we proceed by induction on $n_1$.  The base case is $n_1 = 1,$ which by Theorem \ref{theorem:phi*Recurrence} can be rewritten as:
\begin{align*}
\phi^*&(1, n_2, \dots, n_r) = \sum_{i=1}^r \phi^*(0, n_2, \dots,
n_i)\cdot \phi^*(n_{i+1}, \dots, n_r) \\
&= 2\phi^*(n_2, \dots, n_r) + \sum_{i = 2}^{r-1} \phi_{\{1,i\}}^*(n_2,
\dots, n_r)  = \sum_{S \subset [2,r-1]} \binom{2}{|S_1|}
\phi_{S_1}^*(n_2, \dots, n_r),
\end{align*}
which is the desired formula (note that the $i=1$ and $i=r$ term are identical, and correspond to $S = \emptyset$).

For the general case, suppose that $n_1 > 1.$  Here Theorem \ref{theorem:phi*Recurrence}
 and the inductive hypothesis together imply that
\begin{align}
\label{eq:phi*SRecurrence1}
\phi^*&(n_1, \dots, n_r) = \sum_{i=1}^r \phi^*(n_1-1, n_2, \dots, n_i)
\phi^*(n_{i+1}, \dots, n_r) \\
& = \phi^*(n_2, \dots, n_r) + \sum_{S \subset
[2,r-1]}\binom{n_1}{|S_1|} \phi_{S_1}^*(n_2, \dots, n_r) \notag \\
& \qquad + \sum_{i=2}^{r-1} \sum_{S \subset [2,i-1]}
\binom{n_1}{|S_1|}\phi_{S_1}^*(n_2, \dots, n_i) \phi^*(n_{i+1}, \dots,
n_r), \notag
\end{align}
where we have again separated the $i = 1$ and $i=r$ terms.  Since $i < r$ in the last
sum above, Definition \ref{defn:phi*S} implies that
\begin{equation*}
\phi_{S_1}^*(n_2, \dots, n_i) \phi^*(n_{i+1}, \dots, n_r) = \phi_{S_1
\cup \{i\}}^*(n_2, \dots, n_r),
\end{equation*}
so (\ref{eq:phi*SRecurrence1}) becomes
\begin{align}
\label{eq:phi*SRecurrence2}
\phi^*(n_2, \dots, n_r) & + \sum_{S
\subset [2,r-1]}\binom{n_1}{|S_1|}
\phi_{S_1}^*(n_2, \dots, n_r) \\
& + \sum_{\emptyset \neq S \subset [2,r-1]}
\binom{n_1}{|S_1|-1}\phi_{S_1}^*(n_2, \dots, n_r). \notag
\end{align}
Combining like terms and adding binomial coefficients in (\ref{eq:phi*SRecurrence2}) gives
\begin{equation*}
\phi^*(n_1, \dots, n_r) = \sum_{S \subset [2,r-1]} \binom{n_1 + 1}{|S_1|}\phi_{S_1}^*(n_2,
\dots, n_r).
\end{equation*}
\end{proof}

Theorem \ref{theorem:phi*Recurrence1} can also be interpreted (and proved) combinatorially by examining the structure of the pairings counted by $\phi^*(n_1, \dots, n_r) = \phi(1^{n_1}0^{n_1}\cdots1^{n_r}0^{n_r}).$  Specifically, consider the initial run $1^{n_1}$; the first $1$ must pair with a $0$ contained in some run $0^{n_{i_1}}$.  Define $t_1$ such that the first $t_1$ $1$s all pair to the $(i_1)$-th run of $0$s, but the $(t_1 + 1)$-th $1$ pairs to a different run $0^{n_{i_2}}$.  Continue similarly, so that the $(t_{j-1} + 1)$-th through $(t_j)$-th $1$s pair to the $(i_j)$-th run of $0$s until we reach the $(t_s)$-th $1$, after which all subsequent $1$s pair to the first run of $0$s (if there are no such pairings, then $t_s = n_1$).
\begin{figure}[htbp]
\begin{center}
	\includegraphics[width=0.8\textwidth]{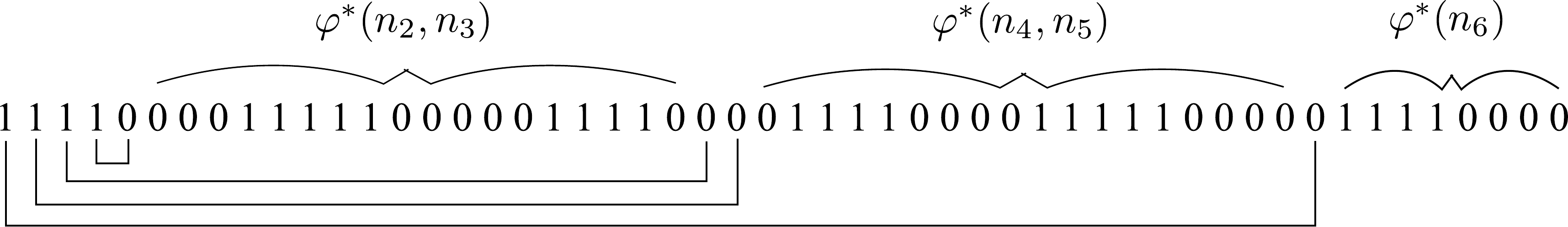}
\caption{Using the notation in the paragraph preceding \eqref{eq:simrecur}, in
    this figure we have $(n_1, \ldots, n_r) = (4,5,4,4,5,4)$, and $(i_s, \ldots, i_1) = (3,5)$ and $\{t_1,
    \ldots, t_s\} = \{1,3\}$.  Notice that we are also using the rotational
    symmetries of $\phi$ here.  That is: the segment giving $\phi^*(n_2, n_3)$,
    equal to $\phi^*(4,4)$ in this case, is actually given by $\phi(3,4,4,4,1)$.  By
    the rotational symmetries of $\phi$, this is equal to $\phi^*(4,4)$.}
\label{fig:intuitivepicture}
\end{center}
\end{figure}
Thus we have a set $S = \{i_s, i_{s-1}, \dots, i_1\} \subset [2,r],$ and also
integers $1 \leq t_1 < t_2 < \dots < t_s \leq n_1,$ and by rotation as in Corollary
\ref{corollary:phisymmetry} (see Figure \ref{fig:intuitivepicture}) there are $\phi_{S_1}^*(n_2, \dots, n_r)$ ways to pair the remaining portions of the bitstring (Definition \ref{defn:phi*S} is extended in the obvious way if $r \in S$).  This completes the proof of an equivalent form of Theorem \ref{theorem:phi*Recurrence1}:
\begin{equation}\label{eq:simrecur}
\phi^*(n_1, \dots, n_r) = \sum_{S \subset [2,r]}
\binom{n_1}{|S|}\phi_{S_1}^*(n_2, \dots, n_r).
\end{equation}

\subsection{Tree expansion for $\phi^*$}
\label{subsection:Treephi*}
We now find an exact formula for $\phi^*$ by further unwinding the recurrence of Theorem \ref{theorem:phi*Recurrence1} and introducing a natural tree structure that describes the iteration.  First, we recall the theory of Catalan sequences and trees.
\begin{defn}
\label{definition:Catalan}
A {\it Catalan degree sequence of order} $r \geq 1$ is an $r$-tuple $\dd := (d_1, d_2, \dots, d_r)$ of nonnegative integers such that
\begin{align*}
d_1 + \dots + d_i & \geq i \qquad \mathrm{if} \; 1 \leq i \leq r-1, \\
d_1 + \dots + d_r & = r-1.
\end{align*}
Denote the set of all such sequences of order $r$ by $\Dr$.
\end{defn}

It is well-known \cite[Section 5.3]{Stanley} that Catalan sequences of order $r$ are precisely the
degree sequences of rooted plane trees with $r$ vertices (all of the trees that we
consider in this work are assumed to be planar).  Given such a tree $T$, the
correspondence follows by labeling the vertices with $1, \dots, r$ according to the
{\it depth-first traversal}.  Such a traversal is defined by visiting the vertices
following a linear ordering that satisfies $v < w$ whenever $w$ is a child of $v$.
The associated Catalan sequence is then given by $d_i = \#\{\text{children of vertex }
i\}$, and for a tree $T$, we denote this {\it tree degree sequence} by $\dT$.
Furthermore, we denote the set of all rooted plane trees with $r$ vertices by $\Tr$.
We canonically display the children of each vertex ordered from left to right, so
that the depth-first traversal proceeds clockwise around the tree.  Examples are
given in Figure \ref{fig:planetrees}.
\begin{figure}[htbp]
\begin{center}
	\includegraphics[width=0.8\textwidth]{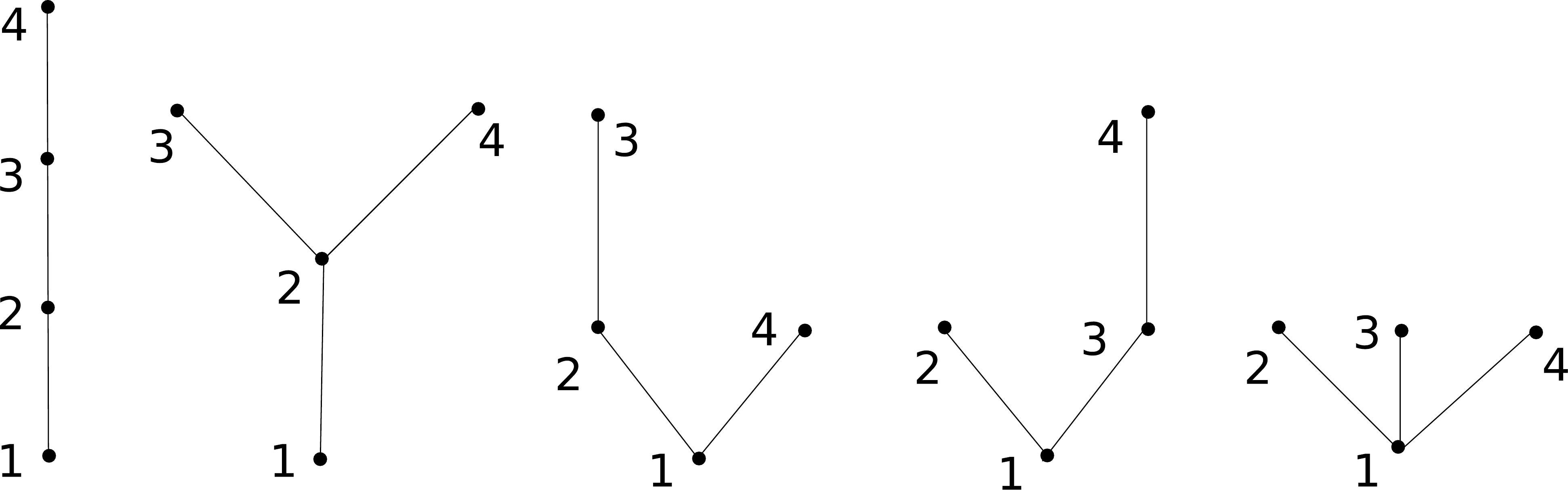}
\caption{The trivial labeling of the five plane trees with four vertices.  Their
respective degree sequences are $(1,1,1,0), (1,2,0,0), (2,1,0,0), (2,0,1,0)$ and
$(3,0,0,0)$.}
\label{fig:planetrees}
\end{center}
\end{figure}

If the root vertex is removed, then what remains is a forest of $d_1$ (sub)trees; we denote these by $T_1, \dots, T_{d_1}$. This root-removal procedure immediately implies the following standard result.
\begin{lemma}
\label{lemma:dforest}
If $\dd \in \Dr$, then there are unique indices $1 = i_0 < i_1 < \dots < i_{d_1} = r$ such that $(d_{i_j+1}, \dots, d_{i_{j+1}}) \in \D_{i_{j+1} - i_j}$ for $0 \leq j \leq d_1 - 1.$
\end{lemma}
\noindent We can equivalently define the indices $i_j$ as satisfying the condition $T_j \in \T_{i_{j} - i_{j-1}}.$

\medskip

We postpone further discussion of the combinatorics of these sequences until after the proof of the following special case of the general formula for $\phi^*$.
\begin{theorem}
\label{theorem:phi*increasing}
If $0 \leq n_1 \leq n_2 \leq \dots \leq n_r,$ then
\begin{equation*}
\phi^*(n_1, n_2, \dots, n_r) = \sum_{\dd \:  \in \: \Dr} \; \prod_{i=1}^r \binom{n_i + 1}{d_i}.
\end{equation*}
\end{theorem}
\begin{proof}
We proceed by induction on $r$.  If $r = 1$, the only Catalan degree sequence is $d_1
= 0$, and thus $\phi^*(n_1) = 1$ as in Proposition \ref{prop low-run formulas}.  For the general case we use the recurrence from Theorem \ref{theorem:phi*Recurrence1} and consider the term corresponding to an arbitrary set $S = \{i_1, \dots, i_s\} \subset [2, r-1]$.  Expanding Definition \ref{defn:phi*S} gives
\begin{equation}
\label{eq:phi*S1}
\phi^*_{S_1} = \phi^*(n_2, \dots, n_{i_1}) \phi^*(n_{i_1 + 1}, \dots, n_{i_2}) \cdots \phi^*(n_{i_s + 1}, \dots, n_r),
\end{equation}
and the induction hypothesis now implies that for any $2 \leq j \leq k \leq r,$
\begin{equation}
\label{eq:phi*njnk}
\phi^*(n_j, \dots, n_k) = \sum_{\dd \:  \in \: \D_{k-j+1}} \; \prod_{i=j}^{k} \binom{n_{i} + 1}{d_i},
\end{equation}
where we have renumbered the indices of the $d_i$ so that they correspond to the $n_i$'s.  Combining (\ref{eq:phi*njnk}), (\ref{eq:phi*S1}), and Theorem \ref{theorem:phi*Recurrence1} (note that by assumption $n_j \leq n_k$ for $j < k$, so the leading $n_{i_j}$s are always minimal) yields the formula
\begin{align}
\label{equation:phi*d'}
\phi^*(n_1,n_2, \dots, n_r) & = \sum_{S \subset [2,r-1]}
\binom{n_1+1}{|S_1|} \phi_{S_1}^*(n_2, \dots, n_r) \\
& = \sum_{S \subset [2,r-1]} \binom{n_1+1}{|S_1|}
\sum_{\dd'} \; \prod_{i=2}^{r} \binom{n_i + 1}{d_i}. \notag
\end{align}
The inner sum is over all $\dd' = (d_2, \dots, d_r)$ that are a concatenation of $s+1$ Catalan sequences
\begin{align*}
(d_2, \dots, d_{i_1}) & \in \D_{i_1 -1} \\
(d_{i_1 + 1}, \dots, d_{i_2}) & \in \D_{i_2 - i_1} \\
& \vdots \\
(d_{i_{s-1}+1}, \dots, d_r) & \in \D_{r-i_{s-1}}.
\end{align*}
We finish the proof by setting $d_1 = |S_1| = s+1;$ Lemma \ref{lemma:dforest} then implies that $\dd = (d_1, \dots, d_r) \in \Dr,$ so (\ref{equation:phi*d'}) is equivalent to the claimed formula.
\end{proof}

We can equivalently write the formula in Theorem \ref{theorem:phi*increasing} as a sum over trees:
\begin{equation}
\label{eq:phi*tree}
\phi^*(n_1, n_2, \dots, n_r) = \sum_{T \:  \in \: \Tr} \; \prod_{i=1}^r \binom{n_i + 1}{d_i},
\end{equation}
where $d_i$ is the degree of vertex $i$ in $T$.  Again, this formula can also be proven using combinatorial arguments on noncrossing
pairings.  The summand corresponding to some tree $T \in \Tr$ counts all of the
pairings on the bitstring $1^{n_1} 0^{n_1} \dots 1^{n_r} 0^{n_r}$ with {\em pairing
structure} $T$.  This structure describes any noncrossing pairing on $r$ runs by a
tree with $r$ vertices that encodes the pairings at the level of runs rather than
individual bits.  It is defined inductively as follows: if the first run of $1$s
pairs to the $i_1$-th, $i_2$-th, \dots, and $i_{d_1 -1}$-th runs of $0$s, with $i_j > 1$ for all
$j$, then subtree $T_1$ consists of vertices $\{2, \dots, i_1\}$, $T_2$ consists of
$\{i_1 + 1, \dots, i_2\}$, \dots, and $T_{d_1}$ consists of vertices $\{i_{d_1-1} +
1, \dots, r\}$.  The exact structure of subtree $T_j$ is inductively determined by the induced pairing on an appropriately rotated substring.

\medskip

In particular, suppose that the $a$-th $1$ in $1^{n_1}$ is the final $1$ that pairs to $0^{n_{i_j}}$.  Due to height considerations, this $1$ must pair to the $(n_{i_j}-a+1)$-th $0$ in $0^{n_{i_j}}$, and the $(a+1)$-th $1$ then pairs to the $(n_{i_{j-1}}-a)$-th $0$ in $0^{n_{i_{j-1}}}$.  Thus the substring that is paired ``locally'' (due to the noncrossing condition) is
\begin{equation*}
0^a 1^{n_{i_{j-1}+1}}0^{n_{i_{j-1}+1}} \dots 1^{n_{i_j}} 0^{n_{i_j} -a}.
\end{equation*}
This is rotationally equivalent to $1^{n_{i_{j-1}+1}}0^{n_{i_{j-1}+1}}\dots
1^{n_{i_j}}0^{n_{i_j}}$, and we can now proceed inductively in the construction of
$T_j$ by considering which runs of $0$s are paired to the first run $1^{n_{i_{j-1}+1}}$.
\begin{figure}[htbp]
\begin{center}
	\includegraphics[width=0.8\textwidth]{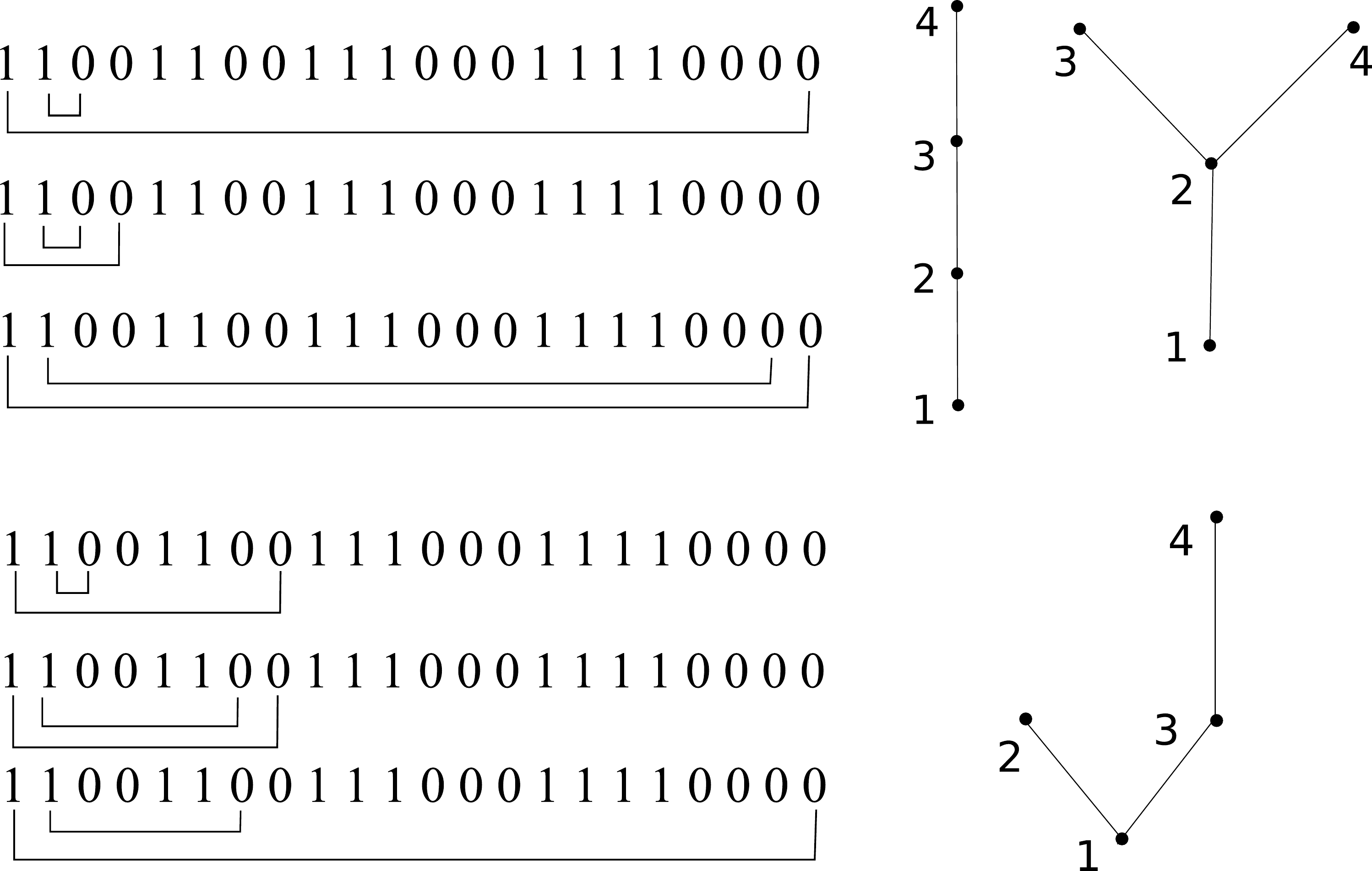}
\caption{In this figure, we have $(n_1, \ldots, n_r) = (2,2,3,4)$.  The first three
types of pairings corresponding to trees, with root vertices of degree 1, immediately
to their right.  The second set of pairings correspond to the tree immediately to
their right.}
\label{fig:pairingtrees}
\end{center}
\end{figure}

For example, in the diagram in Figure \ref{fig:pairingtrees} the first three types of
pairings correspond to the two plane trees on four vertices with root having degree
1.  Notice that in these three types of pairings, what remains is a pairing on the string $\mx{S} =
110011100011110000$.  Inductively, we compute that the number of such pairings is 15
and thus the total number of the first three types of pairings in Figure
\ref{fig:pairingtrees} is 45.  A simple computation shows that the terms of
\eqref{eq:phi*tree} corresponding to plane trees of degree 1 are
\begin{align*}
	{n_1 + 1 \choose d_1}  {n_2 + 1 \choose d_2}  {n_3 + 1 \choose d_3}
	&+ {n_1 + 1 \choose d_1} {n_2 + 1 \choose  d_2}\\ &=
	{2 + 1 \choose 1}  {2 + 1 \choose 1}  {3 + 1 \choose 1} + {2 + 1
	\choose 1} {2 + 1 \choose 2} \\
	&= 45
\end{align*}
Similarly, the second set of three pairings correspond to the tree to their right.
Notice that the number of strings of each type in this second category is $1
\times 4$ (one pairing for the string $1100$ and four pairings for the string
$11100011110000$).  Again, these strings are paired locally, due to the pairing being
noncrossing.  Thus, there should be 12 pairings of this type.  The term in \eqref{eq:phi*tree} corresponding to this tree is
\begin{equation*}
	{2 + 1 \choose 2} \cdot {3 + 1 \choose 1} = 12,
\end{equation*}
in agreement with the number of pairings.

It is the version of Theorem \ref{theorem:phi*increasing} in (\ref{eq:phi*tree}) that we generalize to the case of arbitrary $n_i$, using certain labelings on the pairing structure trees to encode the recursion for $\phi^*$.
\begin{defn}
\label{defn:Tlabeling}
If $T \in \Tr$, the {\it labeling} of $T$ by positive integers $(a_1, \dots, a_r)$ is determined by the following recursive procedure:
\begin{enumerate}
 \item Cyclically rotate $(a_1, \dots, a_r)$ so that $a_1$ is minimal.
 \item Label the root vertex with $a_1$.
 \item Label the subtree $T_j$ by $(a_{i_j+1}, \dots, a_{i_{j+1}})$, with $i_j$ as in Lemma \ref{lemma:dforest}.
\end{enumerate}
The {\em labeling of $T$ by a permutation} $\sigma \in S_r$ is defined to be the labeling of $T$ by the sequence $(\sigma(1), \dots, \sigma(r))$.  Conversely, define the {\em inverse permutation} of $T$ labeled by $\sigma$ as the $\sigma_T \in S_r$ satisfying  $$\sigma_T(i) = \text{label of vertex} \; i \; \text{in} \; T.$$
\end{defn}
\begin{remark} Note that $\sigma = \text{id}$ corresponds to the usual depth-first
	numbering of $T$ (so vertex $i$ is also labeled with $i$), and in this case
	$\sigma_T = \text{id}$ as well.  Figure \ref{fig:nontrivtreeexamp} gives an
	example of a nontrivial labeling.
\end{remark}
\begin{figure}[htbp]
\begin{center}
	\includegraphics[width=0.28\textwidth]{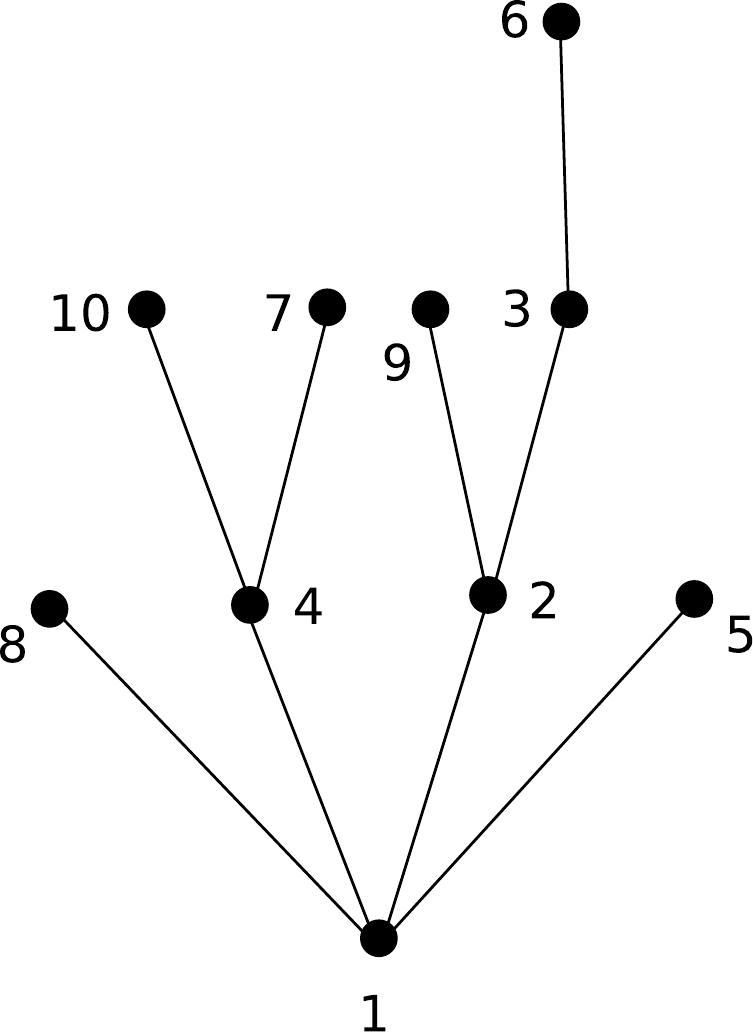}
	\caption{A tree labeled with the permutation $\sigma =
	(2,9,6,3,5,1,8,7,4,10)$.}
\label{fig:nontrivtreeexamp}
\end{center}
\end{figure}

With these definitions and notation, we can now state and prove the general formula for $\phi^*$.
\begin{theorem}
\label{theorem:phi*formula}
If $(n_1, \dots, n_r) = (n'_{\sigma(1)}, \dots, n'_{\sigma(r)}),$ with $0 \leq n'_1 \leq \dots \leq n'_r,$ then
\begin{equation*}
\phi^*(n_1, n_2, \dots, n_r) = \sum_{T \:  \in \: \Tr} \; \prod_{i=1}^r \binom{n'_{\sigma_T(i)} + 1}{d_i}.
\end{equation*}
\end{theorem}
\noindent Before proving this formula, we illustrate its usefulness by writing down
the resulting formula for $\phi^*(n'_4, n'_1, n'_3, n'_2)$.  Figure
\ref{fig:planetreesnontriv} shows all trees in $\T_4$ labeled with the sequence
$\sigma = (4,1,3,2)$.  Adding up all of these terms gives
\begin{figure}[htbp]
\begin{center}
	\includegraphics[width=0.7\textwidth]{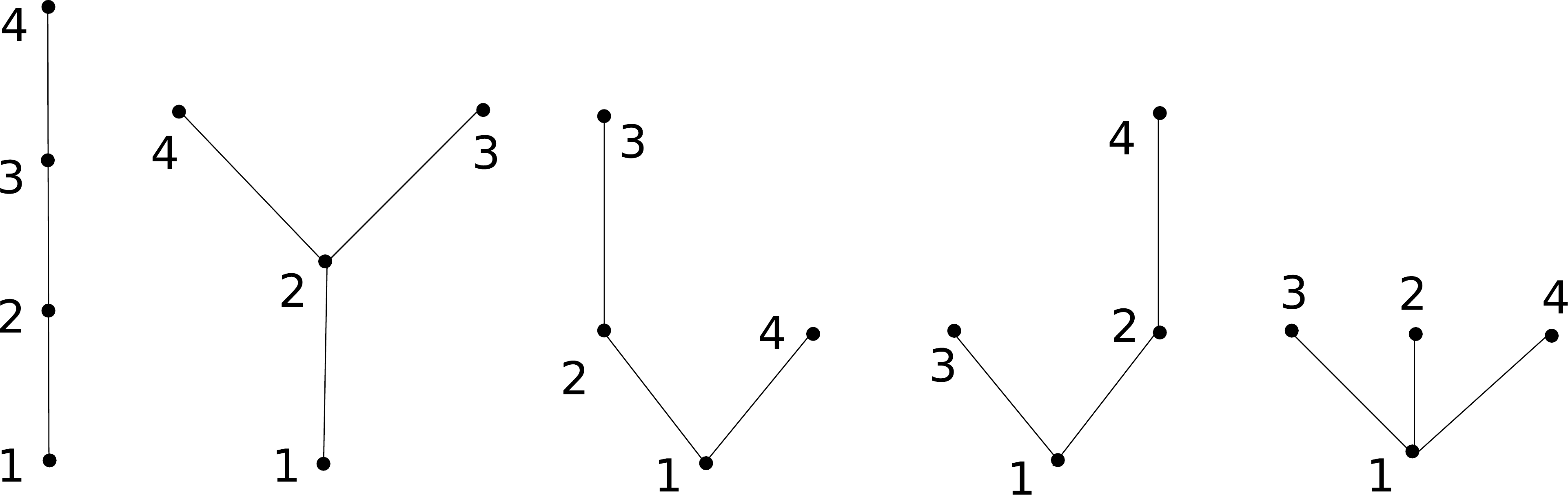}
	\caption{The five plane trees labeled with the permutation $\sigma = (4,1,3,2)$.}
\label{fig:planetreesnontriv}
\end{center}
\end{figure}
\begin{align}
\label{equation:phi*example}
\phi^*(n'_4, n'_1, n'_3, n'_2) = & \binom{n'_1 + 1}{1}\binom{n'_2 + 1}{1}\binom{n'_3 + 1}{1} +
 \binom{n'_1 + 1}{1}\binom{n'_2 + 1}{2} \\
 & + 2 \binom{n'_1 + 1}{2}\binom{n'_2 + 1}{1} +
 \binom{n'_1 + 1}{3}. \notag
\end{align}
\begin{proof}[Proof of Theorem \ref{theorem:phi*formula}]
The argument is very similar to the proof of Theorem \ref{theorem:phi*increasing}, with one key difference: the recursive calls to Theorem \ref{theorem:phi*Recurrence1} require the first argument to be minimal, which was automatically satisfied when the sequence of $n_i$s was assumed to be weakly increasing.  With an arbitrary permutation $\sigma,$ however, this is no longer the case, and thus we are forced to rotate to a minimal $n_i$ in each recursive call.  However, the labeling procedure from Definition \ref{defn:Tlabeling} exactly accounts for these rotations and recursive splits, and therefore the permutation $\sigma_T$ is the necessary adjustment so that the formula holds in general.
\end{proof}

\subsection{The Conjectural Structure of Tree Polynomial Orderings}
\label{section tree ordering}

For notational convenience, we write
\begin{equation}
\label{eq:[n]^k}
[n]^k := \binom{n+1}{k}.
\end{equation}
Furthermore, we call the summands on the right-hand side of Theorem \ref{theorem:phi*formula} the {\it tree polynomials} associated to $\sigma,$ considered as functions of the $n'_i$.  We denote these by
\begin{equation}
\label{eq:pT}
\pT = \pT(n'_1, \dots, n'_r) := \prod_{i = 1}^r \: [n'_{\sigma'(i)}]^{d_i}.
\end{equation}
The {\it degree} of such a polynomial is simply $d_1 + \dots + d_r$.

Theorem \ref{theorem:phi*formula} can now be written compactly as a homogeneous (in degree) sum of tree polynomials,
\begin{equation*}
\phi^*(n'_{\sigma(1)}, \dots, n'_{\sigma(r)}) = \sum_{\dT \: \in \: \Tr} \pT.
\end{equation*}
In practice we will write the products in (\ref{eq:pT}) ordered by increasing $n'$ subscripts.  For example, if $\sigma = (4132) \in S_4$, then the formula from (\ref{equation:phi*example}) is written more compactly as
\begin{equation}
\label{equation:phi*example[]}
\phi^*(n'_4, n'_3, n'_1, n'_2) = [n'_1]^1[n'_2]^1[n'_3]^1 + [n'_1]^1[n'_2]^2 + [n'_1]^2[n'_2]^1 + [n'_1]^2[n'_2]^1 + [n'_1]^3.
\end{equation}

In order to find an upper bound for $\phi^*$ we must compare expressions of the form (\ref{eq:pT}).  Elementary properties of binomial coefficients imply the following two simple tests for inequality that are sufficient for our purposes.
\begin{proposition}
\label{proposition:[n1]}
Suppose that $n_1 \leq n_2$.
\begin{enumerate}
\item
For any $k$, we have
\begin{equation*}
[n_1]^k \leq [n_2]^k.
\end{equation*}
\item
If $k_1 \leq k_2$, then
\begin{equation*}
[n_1]^{k_2} [n_2]^{k_1} \leq [n_1]^{k_1} [n_2]^{k_2}.
\end{equation*}
\end{enumerate}
\end{proposition}

\begin{defn}
Two tree polynomials $p_1$ and $p_2$ are {\it comparable} if there is a string of inequalities using the rules from Proposition \ref{proposition:[n1]} implying either $p_1 \leq p_2$ or $p_2 \leq p_1$.
\end{defn}
\begin{corollary}
\label{corollary:pT1T2}
If $p_{\: T_1}$ and $p_{\: T_2}$ are comparable, and $T_i$ has degree sequence $\dd^i$ for $i = 1, 2$, then $\dd^1 = \dd^2$ as multisets.
\end{corollary}
\begin{remark}  The converse does not hold.  For example, $[n'_1]^1[n'_4]^1$ and $[n'_2]^1[n'_3]^1$ are not comparable by our definition even though the degree multisets are both $\{1, 1\}$.  Furthermore, the terminology is appropriate, since as the $n'_i$ vary under the constraint $n'_1 \leq n'_2 \leq n'_3 \leq n'_4$, either one of the polynomials can be the larger of the two.
\end{remark}

In Section \ref{section:MainProof} we show through different means that $\phi^*$ achieves its maximum value when the arguments are weakly increasing; i.e., for any permutation $\sigma$ and weakly increasing sequence $\{n'_i\}$, we have
\begin{equation}
\label{eq:phi*n'n}
\phi^*(n'_{\sigma(1)}, \dots, n'_{\sigma(r)}) \leq \phi^*(n'_1, \dots, n'_r).
\end{equation}
Our proof there requires us to non-bijectively inject non-crossing pairings into a different set of combinatorial objects, and we therefore do not make direct use of the formula in Theorem \ref{theorem:phi*formula}.  However, numerical evidence suggests that (\ref{eq:phi*n'n}) can be understood and in fact greatly refined through the tree formula as well (recall that the right-hand side of (\ref{eq:phi*n'n}) corresponds to the identity permutation).
\begin{conjecture}
\label{conjecture:pTtau}
Suppose that $0 \leq n'_1 \leq \dots \leq n'_r$.  Then for any permutation $\sigma \in S_r$ there exists a re-ordering $\tau : \Tr \rightarrow \Tr$ such that $\pT$ and $p_{\: \tau(T), 1}$ are comparable, with
\begin{equation*}
\pT(n'_1, \dots, n'_r) \leq p_{\: \tau(T), \text{id}}(n'_1, \dots, n'_r).
\end{equation*}
\end{conjecture}
\begin{example*}
Let $\sigma = (4132)$.  We have already calculated $\phi^*(n'_4, n'_1, n'_3, n'_2)$ in (\ref{equation:phi*example[]}), and Theorem \ref{theorem:phi*increasing} implies that
\begin{equation}
\label{equation:phi*identity}
\phi^*(n'_1, n'_2, n'_3, n'_4) = [n'_1]^1[n'_2]^1[n'_3]^1 + [n'_1]^1[n'_2]^2 + [n'_1]^2[n'_2]^1 + [n'_1]^2[n'_3]^1 + [n'_1]^3.
\end{equation}
Comparing terms in (\ref{equation:phi*example[]}) and (\ref{equation:phi*identity}) shows that Conjecture \ref{conjecture:pTtau} holds in this case with $\tau = \text{id}.$
\end{example*}
\begin{remark}
In light of Corollary \ref{corollary:pT1T2}, Conjecture \ref{conjecture:pTtau} requires that $\tau$ preserve the degree multiset of any tree $T$, and thus $\tau$ respects the partition of $\Tr$ into subsets of trees that all have the same degree multisets.
\end{remark}

Finally, we note that we have verified the truth of Conjecture \ref{conjecture:pTtau} for $r \leq 7$ by using Mathematica to computationally check all permutations in $S_r$, but the general case remains a significant open problem. We conclude by recording several partial results in the direction of the conjecture.
\begin{definition}
A sequence $(a_1, \dots, a_r)$ is {\em unimodal} if there is some $1 \leq k \leq r$ such that
\begin{align*}
a_i \leq a_j \leq a_k \qquad & \text{if} \; i < j \leq k, \\
a_k \geq a_i \geq a_j \qquad & \text{if} \; k \leq i < j.
\end{align*}
A permutation $\sigma \in S_r$ is {\em unimodal} if $(\sigma(1), \dots, \sigma(r))$ is unimodal.
\end{definition}
We will exploit a simple recursive property of unimodal permutations.
\begin{lemma}
\label{lemma:unimodal}
Suppose that $\sigma$ is unimodal.  If $1, 2, \dots, k$ are removed from the sequence $(\sigma(1), \dots, \sigma(r)),$ then $k+1$ is either the left-most or right-most remaining term.
\end{lemma}
\begin{proposition}
\label{proposition:unimodal}
If $\sigma$ is unimodal, then
\begin{equation*}
\phi^*(n'_{\sigma(1)}, \dots, n'_{\sigma(r)}) = \phi^*(n'_1, \dots, n'_r).
\end{equation*}
\end{proposition}
\begin{proof}
We show that the tree expansion for $\phi^*(n'_{\sigma(1)}, \dots, n'_{\sigma(r)})$ from Theorem \ref{theorem:phi*formula} contains every term in Theorem \ref{theorem:phi*increasing}; since both sums are indexed by sets of the same cardinality $C_r$, this proves the claimed equality.  Consider the term $p = [n'_1]^{d_1} \dots [n'_r]^{d_r}$, where $\dd \in \Dr.$  We construct a tree $T$ such that $\pT = p$ by the following procedure (recall Lemma \ref{lemma:unimodal}):
\begin{enumerate}
\item
Create vertex $1$ with $d_1$ children.  Initialize $s = (\sigma(1), \dots, \sigma(r))$, set $i = 2$, and remove $1$ from $s$.
\item
If $i$ is the left-most term in $s$, then label the left-most (relative to the depth-first traversal) leaf with $i$.  Similarly, if $i$ is the right-most term in $s$, label the right-most leaf with $i$.
\item
Add $d_i$ children to vertex $i$.
\item
If $i < r$, remove $i$ from $s$, increment $i$ by $1$, and return to step (2).
\end{enumerate}
This procedure is well-defined since Definition \ref{definition:Catalan} implies that there will always be unlabeled leaves available.  Furthermore, the result is a tree $T$ such that $\pT = p$ as desired.
\end{proof}

\begin{proposition}
If $\sigma$ is not unimodal, then $\phi^*(n'_{\sigma(1)}, \dots, n'_{\sigma(r)}) < \phi^*(n'_1, \dots, n'_r).$
\end{proposition}
\begin{proof}
For such a $\sigma$, Lemma \ref{lemma:unimodal} implies that there exists an $N > 1$ such that there is no consecutive subsequence in $(\sigma(1), \dots, \sigma(r))$ that consists of precisely $\{N, N+1, \dots, r\}$.  The formula in Theorem \ref{theorem:phi*increasing} has a term corresponding to the Catalan sequence whose nonzero degrees are $d_1 = N, d_N = r-N-1$, but Definition \ref{defn:Tlabeling} implies that $\pT \neq [n'_1]^{N}[n'_N]^{r-N-1}$ for all $T$.  Thus Theorem \ref{theorem:phi*formula} implies that $\phi^*(n'_{\sigma(1)}, \dots, n'_{\sigma(r)})$ is less than $\phi^*$ for the identity permutation.
\end{proof}

\section{Catalan Structures and the Proof of Theorem \ref{Main Theorem}}
\label{section:MainProof}

In this section we prove the upper bound for $\phi$ from Theorem \ref{Main Theorem}.  In order to do so, we shift our attention from the exact formulas of Section \ref{section:phi*} to certain closely related {\it Catalan structures}; these are certain classes of labeled trees and bitstrings that will be defined below.  The key difference between these new structures and the non-crossing pairings is that the property of ``word rotation'' becomes locally restricted.  We exploit this by pre-rotating our non-crossing pairings (which have unrestricted word rotation as in Proposition \ref{prop rotation}) and then mapping them injectively to associated labeled trees.  The combinatorics of these trees and bitstrings then quickly lead to our claimed upper bounds.

We could simplify the following exposition somewhat by first appealing to Theorem \ref{theorem:phiphi*} and then restricting our attention to proving bounds only for $\phi^*$.  However, the connections between all noncrossing pairings and Catalan structures has independent combinatorial interest, so we instead present the general case.

\subsection{Noncrossing Pairings and Labeled Trees}

For notational convenience, define $\n := (n_1, \dots, n_r)$ and $\m := (m_1, \dots, m_r)$; throughout this section we require that $\n, \m \in \N_+^r$.  We say that $\n$ is {\it weakly increasing} if $n_i \leq n_{i+1}$ for all $i$. Furthermore, define the bitstring $w(\n, \m) := 1^{n_1} 0^{m_1} \dots 1^{n_r} 0^{m_r},$ and let
\begin{eqnarray*}
\phi(\n,\m) & := & \phi(w(\n, \m)),\\
\phi^*(\n) & := & \phi(\n, \n).
\end{eqnarray*}
(Note that this is a slight change in notation for the integer-argument representation of $\varphi$; we no longer interleave the $\n$ and $\m$.)  We will be particularly interested in the collection of words where $\n$ is fixed.
\begin{definition}
Let $W(\n)$ be the set of all balanced words $w = 0^{m_0} 1^{n_1} 0^{m_1} \ldots 1^{n_r} 0^{m_r}$ with $m_i \geq 0$ for $0 \leq i \leq r$.  Similarly, let $W^*(\n)$ be the set of all words $w = 0^a 1^{n_1} 0^{n_1} \ldots 1^{n_r} 0^{n_r-a}$ where $0 \leq a \leq n_r$.
\end{definition}
\begin{remark*}
The sets $W(\n)$ (respectively $W^*(\n)$) contain all words that are equivalent up to ``local rotation'' to a word of the form $1^{n_1} 0^{m_1} \dots 1^{n_r} 0^{m_r}$ (resp. $1^{n_1} 0^{n_1} \dots 1^{n_r} 0^{n_r}$).
\end{remark*}

For any pairing on a word in $W(\n)$, we will define an associated tree that is
labeled with additional data describing the pairing.  If $d,n \geq 0$ are integers, a
{\em label} of {\em degree} $d$ and and {\em weight} $n$ is a $(d+1)$-tuple $l =
(l_i: 0 \leq i \leq d)$ of non-negative integers with $l_j \geq 1$ for all $0 < j <
d$ and $\sum_{j=0}^d l_j = n$.  Elementary arguments give the number of labels of degree $d$ and
weight $n$ as ${n + 1 \choose d}$.  Recall from Section \ref{subsection:Treephi*} that a tree $T \in \Tr$ has an associated degree sequence $\dT$.
\begin{definition}
Let $\LT(\n)$ be the set of all pairs $(T,L)$ where $T \in \Tr$ and $L$ is a labeling function on the vertices of $T$ that assigns vertex $i$ to a label $L_i$ of degree $d_i$ and weight $n_i$.  Each such $(T,L)$ is called a {\em labeled tree}.
\end{definition}
Since there are ${ n_i +1 \choose d_i}$ ways to choose a label $L_i$ for vertex $i$ in $T$, we immediately obtain a closed formula for the number of labeled trees.
\begin{proposition}
\label{proposition:LTformula}
For any $\n$,
\begin{equation*}
|\LT(\n)| = \sum_{T \in \Tr} \prod_{i=1}^r {n_i + 1\choose d_i}.
\end{equation*}
\end{proposition}
Comparing this to Theorem \ref{theorem:phi*formula}, we observe that here there is no dependence on the relative magnitude of the $n_i$'s.  This difference is the key reason that translating noncrossing pairings to labeled trees results in such a simplification of the problem of finding bounds.  The next result describes an important relation between noncrossing pairings and labeled trees, along with an immediate implication on cardinalities.
\begin{theorem}
\label{theorem:WLT}
For all $\n$ , and $w \in W(\n)$ there is an injection $T_{w,\n} : NC_2(w) \rightarrow \LT(\n)$.   If $\n$ is weakly increasing and $w \in W^*(\n)$, then $T_{w,\n}$ is a bijection.
\end{theorem}

\begin{corollary}
\label{corollary:phiLT}
For all $\n,\m$, we have $\phi(\n,\m) \leq |\LT(\n)|$.  If $\n$ is weakly increasing, then $\phi^*(\n) = |\LT(\n)|$.
\end{corollary}

\begin{proof}[Proof of Theorem \ref{theorem:WLT}] The first half of the proof is devoted to the construction of $T = T_{w,\n}(\pi) \in \LT(\n)$ for a given $\n$, a word $w \in W(\n)$, and a pairing $\pi \in NC_2(w)$.  The second half of the proof verifies that the definition is bijective by reversing the map.  We begin by finding certain substrings in $w$ that correspond to subtrees in $T_{w, \n}(\pi)$.

Let $s \geq 0$ be the number of leading $0$s of $w$, and rotate by $s$ so that the resulting word $w' = \Rot_s(w)$ begins with $1^{n_1}$. Denote the rotated pairing by $\pi' = \Rot_s(\pi) \in NC_2(w')$.  Let $X := \{1, \dots, n_1\}$ be the set of positions in $w'$ occupied by the block $1^{n_1}$, and define $Y$ to be the set of positions that are paired to positions in $X$ by $\pi'$.   Let $w'_1, \ldots, w'_d$ be the set of all non-empty substrings of $w'$ that lie between two circularly consecutive positions in $X \cup Y$, listed in the order that they appear in $w'$.  We have defined the decomposition
\begin{equation}
w' = 1^{n_1} 0^{\ell_0} w'_1 0^{\ell_1} \ldots 0^{\ell_{d-1}}  w'_d 0^{\ell_d},
\end{equation}
where the $0^{\ell_i}$ represent the positions in $Y$, and satisfy $\ell_i \geq 1$ for $0 < j < d$, while $\ell_0, \ell_d \geq 0.$  Moreover, $\sum_{j=0}^d \ell_j = n_1$, so $\Ell=(\ell_0, \ldots, \ell_d)$ is a label of degree $d$ and weight $n_1$ (note that $d=0$ iff $w'= 1^{n_1} 0^{n_1}$).

Let $B$ be the set of $r-1$ distinct blocks $1^{n_j}$ in $w'$ for $2 \leq j \leq r$.  All positions of $w'$ not belonging to the $w'_i$ are $0$'s or members of the block $1^{n_1}$, and no two $w'_i$ are cyclically (i.e.\ ``circularly'') adjacent, so each member of $B$ must be contained in a unique $w'_i$.  Since $\pi$ is non-crossing it induces $\pi'_i \in NC_2(w'_i)$ for all $1 \leq i \leq d$, and as a consequence, each $w'_i$ is balanced. Since each $w'_i$ is also non-empty, it must contain at least one $1$, and hence at least one block in $B$.  Thus $d \leq r-1$, so we set $d_1 = d$ and assign the root label $L_1 = \Ell$ in $T_{w,\n}(\pi)$.

For $1 \leq i \leq d$, let $j_i$ be the maximum $2 \leq j \leq r$ such that $1^{n_j}$ is contained in $w'_i$.  Thus we have a uniquely defined sequence $1 = j_0 < \cdots < j_d = r$ so that $w'_i$ contains blocks $1^{n_t}$ for $j_{i-1}+1 \leq t \leq n_j$. We have $w'_i \in W(\n_i)$ where $\n_i := (n_{j_{i-1}+1}, \ldots, n_{j_i})$.  The definition of $T = T_{w,n}(\pi)$ is now completed recursively: for $1 \leq i \leq d$, the subtree $T_i$ that is attached to the $i$-th child of the root defined to be $T_{w'_i,\n_i}(\pi'_i)$.

The remainder of the proof is spent verifying the following claims:
\begin{enumerate}
 \item $T_{w,\n}(\pi) \in \LT(\n)$,
 \item $\pi$ can be reconstructed from $\n$, $w$, and $T_{w,\n}(\pi)$,
 \item If $\n$ is weakly increasing, $w \in W^*(\n)$ and $T \in \LT(\n)$, then it is possible to construct $\pi \in NC_2(w)$ so that $T_{w,\n}(\pi) = T$.
\end{enumerate}
The first two show that $T_{w, \n}$ is injective, while the last shows that it is bijective when $\n$ is weakly increasing.  We proceed by induction on $r$, as the statements are certainly true for $r=1$.  In this case, each $w \in W(\n) = W^*(\n)$ has only one possible pairing $\pi \in NC_2(w)$, and $T_{w,\n}(\pi)$ produces the unique member of $LT(\n)$, a tree consisting only of the root, which has the label $(n_1)$.  Now suppose that $r > 1.$ \\

\noindent (1) The preceding construction produces the tree $T_{w,\n}(\pi)$, whose root is label $\Ell$ has weight $n_1$ and degree $d_1$.  The root has $d_1$ children, and the subtree rooted at the $i$th child, $T_{w'_i,\n_i}(\pi'_i)$, is a member of $LT(\n_i)$ by the inductive hypothesis.  Since $\n$ is simply the concatenation of $n_1$ and all of the $\n_i$ ($1 \leq i \leq d_1$), we have $T_{w,\n}(\pi) \in LT(\n)$, as labeled trees have the obvious recursive structure. \\

\noindent (2) Let $s$ be the number of leading $0$s of $w$ and $w' = \Rot_s(w)$.   For $1 \leq i \leq d$, let $T_i$ be the subtree of $T$ rooted at the $i$th child.  Set $1 = j_0 < \cdots < j_d = r$ so that the vertices of $T_i$ are $[j_{i-1}+1,j_i]$, and set $\n_i = (n_{j_{i-1}+1}, \ldots, n_{j_i})$ as before.

We now recover the strings $w'_i$ from $w'$ by using each $\n_i$ to calculate that $w'_i$ has length
\begin{equation*}
2 \sum_{s = j_{i-1}+1}^{j_i} n_s.
\end{equation*}
Furthermore, the interspersed runs $0^{\ell_i}$ are determined by $T$'s root label $L_1$.  By induction, we can then reconstruct each $\pi'_i$ from $\n_i$,  $w'_i$ and $T_i$.  Finally, the pairings of $1^{n_1}$ in $w'$ are uniquely determined by the fact that $\pi'$ is non-crossing, so we can reconstruct $\pi = \Rot_{-s}(\pi')$.

\noindent (3)  To achieve our goal of finding $\pi$ such that $T_{w',\n}(\pi') = T$, we certainly must have the correct root label $L_1$.  This implies that if $w' = 1^{n_1} 0^{\ell_0} w'_1 0^{\ell_1} \ldots 0^{\ell_{d-1}} w'_d 0^{\ell_d}$, then the block $1^{n_1}$ is paired to all of the $0^{\ell_i}$'s.  The heights of these $0$'s therefore must be $n_1$ through $1$ consecutively, which is always possible in $w$ since $\n$ is weakly increasing and $w \in W^*(\n)$ is symmetric.  Furthermore, this choice ensures that $w'_i \in W^*(\n_i)$, and
\begin{equation}
\Rot_{\ell_i + \dots + \ell_d}(w'_i) = (1^{n_{j_{i-1}+1}}, 0^{n_{j_{i-1}+1}}, \dots, 1^{n_{j_i}}, 0^{n_{j_i}}).
\end{equation}
By induction, $\pi'_i \in NC_2(w'_i)$ can now be found so that $T_{w'_i,\n_i}(\pi'_i) = T_i$ for all $i$, completing the inverse map.
\end{proof}

We can now prove an upper bound for $\phi(\n, \m)$ by proving one for $|\LT(\n)|$, which we accomplish by bijectively translating labeled trees to the words defined in the next section.

\subsection{Catalan words}

The following definition is a simple generalization of combinatorics commonly associated with the Fuss-Catalan numbers.
\begin{definition}
A word $1^{n_1} 0^{m_1} \dots 1^{n_r} 0^{m_r} \in W(\n)$ is {\it Catalan} if for all $i$,
\begin{align*}
n_1 - m_1 + n_2 - m_2 + \dots + n_i - m_i \geq 0 \qquad & \text{if} \; 1 \leq i \leq r-1, \\
n_1 - m_1 + n_2 - m_2 + \dots + n_r - m_r = 0.
\end{align*}
The set of all Catalan words in $W(\n)$ is denoted by $\CF(\n).$
\end{definition}
\begin{remark*}
Catalan words are commonly associated with {\it Dyck paths} in the Cartesian plane: the $1$s correspond to unit steps in the $x$-direction, the $0$s correspond to unit steps in the $y$-direction, and the resulting paths all travel from $(0,0)$ to $(N, N)$ while remaining under the line $y = x.$
\end{remark*}

Each Catalan word admits a canonical noncrossing pairing with a number of special properties that we will use to construct a correspondence to labeled trees.
\begin{definition}
For every $w \in \CF(\n)$ the {\em first return pairing} $\pi_0(w) \in NC_2(w)$ is given by
\begin{equation*}
\pi_0(w) := \{ (i,j) : w_i = 1, j = \min(j': j' > i, h_w(j') = h_w(i))\}.
\end{equation*}
\end{definition}
This pairing always exists and is well-defined since for any Catalan word $w$, the
first $1$ has height $1$, there are no $0$s with height less than $1$, and if a
chosen $1$ has height $h$, the next character with height $h$ must be a $0$ (see
Figure \ref{fig:firstret}).
\begin{figure}[htbp]
\begin{center}
	\includegraphics[width=0.7\textwidth]{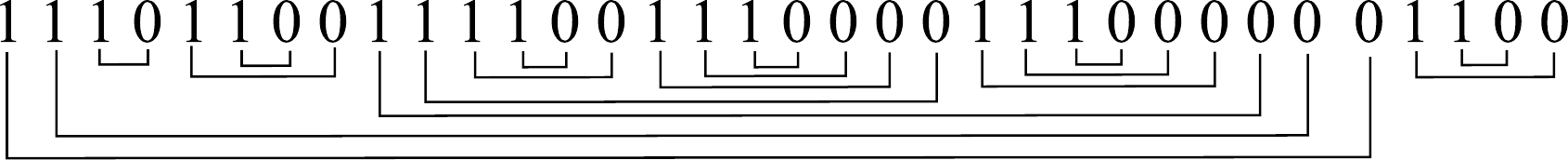}
	\caption{The first return pairing for the string
	$1^301^20^21^40^21^30^41^30^61^20^2$.}
\label{fig:firstret}
\end{center}
\end{figure}

\begin{theorem}
\label{theorem:CFLT}
For all $\n$, the map $f_{\n} : \CF(\n) \rightarrow \LT(\n)$ defined by $f_{\n}(w) := T_{w,\n}(\pi_0(w))$ is a bijection.
\end{theorem}
\begin{proof}
We inductively construct the inverse mapping $g_{\n}: \LT(\n) \rightarrow \CF(\n)$.  If $k=1$, there is only one tree in $\LT(\n)$, a single vertex with label $(n_1)$, and $g_{\n}$ maps it to the only word in $\CF{\n}$, namely $w = 1^{n_1} 0^{n_1}$.  Thus $f_{\n}$ and $g_{\n}$ are mutual inverses.

If $k > 1$, we define
\begin{equation}
g_{\n}(T) := 1^{n_1} 0^{\ell_0} g_{\n_1}(T_1) 0^{\ell_1} \ldots 0^{\ell_{d-1}} g_{\n_d}(T_d) 0^{\ell_d}
\end{equation}
where $\Ell  = (\ell_0, \ldots,\ell_d)$ is the label of the root, $n_1$ is the weight of that label, $T_i$ the subtree of $T$ rooted at the $i$th child of the root, and $\n_i = (n_{j_{i-1}+1}, \ldots, n_{j_i})$ is again the sequence of label weights of the vertices of $T_i$ listed in clockwise order.  By induction, we may assume each $g_{\n_i}(T_i) \in CF(n'_i)$.  Therefore $g_{\n}(T) \in \CF(\n)$.

We now complete the inductive step by showing that $f_{\n}$ and $g_{\n}$ are mutual inverses.  If we first apply $g_{\n}$ to a given $T$, the fact that the $g_{\n_i}(T_i)$ are Catalan implies that the $n_1$ $0$s represented by the $0^{\ell_i}$s in $g_{\n}(T)$ are the first $0$'s in $w$ that occur at heights $n_1$ through $1$ (in order).  These $0$'s are paired by $\pi_0(g_{\n}(T))$ to the $1$s at the corresponding heights in $1^{n_1}$.  Thus the tree $T' = f_{\n}(g_{\n}(T)) = T_{w,\n}(\pi_0(w))$ has a root with $d$ children that is labeled by $\Ell$, and $T'_i$, the subtree of the $i$th child, is $f_{\n_i}(g_{\n_i}(T_i))$.  Note that there is no rotation necessary in the recursive calls to $T_{w, \n_i}$, since $g_{\n_i}(T'_i)$ is inductively a Catalan word and thus begins with a string of $1$s.  Appealing to the inductive hypothesis finally gives $f_{\n_i}(g_{\n_i}(T_i)) = T_i$, so $f_{\n}(g_{\n}(T)) = T' = T$.

On the other hand, if we begin by applying $f_{\n}$ to $w \in \CF(\n)$, then we must first write
\begin{equation}
\label{equation:ww'}
 w = 1^{n_1} 0^{l_0} w'_1 0^{\ell_1} \ldots 0^{\ell_{d-1}} w'_d 0^{\ell_d},
\end{equation}
where the $0^{l_i}$ are the first $0$s in $w'$ of heights $1$ through $n_1$ and the $w'_i$ are non-empty words in $\CF(\n_i)$; as usual, $\n_i = (n_{j_{i-1}+1}, \ldots, n_{j_i})$ for the indices $1 = j_0 < \cdots < j_d = k$ that mark the ends of the $w'_i$s.  Since $\pi_0(w)$ pairs each of the $0$'s in the $0^{\ell_i}$ with the $1$ of the corresponding height in $1^{n_1}$, we obtain $T = f_{\n}(w) = T_{w,\n}(\pi_0(w))$, which is a tree with $d$ children whose root has label $\Ell$. The subtree $T_i$ rooted at the $i$-th child is $T_{w'_i,\n_i}(\pi_0(w'_i)) = f_{\n_i}(w'_i)$.  Thus
\begin{equation}
\label{equation:Wgn}
W := g_{\n}(T) = 1^{n_1} 0^{\ell_1} g_{\n_i}(T_1) 0^{\ell_2} \ldots 0^{\ell^{d-1}} g_{\n_d}(T_d) 0^{\ell_d}.
\end{equation}
By induction, we may assume that $g_{\n_i}(T_i) = g_{\n_i}(f_{\n_i}(w'_i)) = w'_i$.  Combining this with (\ref{equation:ww'}) and (\ref{equation:Wgn}) implies that $g_{\n}(f_{\n}(w)) = W = w$, completing the proof.
\end{proof}

\begin{definition}
If $a,b \in \Te^r$, we say that $a$ {\em dominates} $b$, denoted $a \succcurlyeq b$ if
\begin{equation*}
\sum_{i=1}^j a_i \geq \sum_{i=1}^j b_i
\end{equation*} for all $1 \leq j \leq r$.
\end{definition}
\noindent We will make great use of the simple observation that domination is transitive.
\begin{proposition}
\label{proposition:dom}
If $a \succcurlyeq b$ and $b \succcurlyeq c,$ then $a \succcurlyeq c.$
\end{proposition}

We have already implicitly used the notion of domination in defining Catalan words, as $w(\n,\m) \in \CF(\n)$ if and only if $\n$ dominates $\m$ and $\sum_{i=1}^r n_i  = \sum_{i=1}^r m_i$.  In fact, this relation runs much deeper, as the space of Catalan words also satisfies a very important ordering property with respect to domination.
\begin{lemma}
\label{lemma:CFnCFn'}
For all $r \geq 1$ and $\n,\n' \in \N_+^r$, if $\n' \succcurlyeq \n$ then
\[|\CF(\n)| \leq |\CF(\n')|.\]
\end{lemma}
\begin{proof}
Set $N := \sum_{i = 1}^r n_i$ and $N' := \sum_{i = 1}^r n'_i$, and define the difference $D := N' - N \geq 0$ (nonnegativity follows from the fact that $\n' \succcurlyeq \n$).  Define an injection $\CF(\n) \rightarrow \CF(\n')$ by sending a word $w(\n, \m) \in \CF(\n)$ to $w(\n', \m')$, where
\begin{equation*}
\m' = (m'_1, \dots, m'_{r-1}, m'_r) := (m_1, \dots, m_{r-1}, m_r + D) = \m + (0, \dots, 0, D).
\end{equation*}
Proposition \ref{proposition:dom} implies that $\n' \succcurlyeq \n \succcurlyeq \m$, which when combined with the fact that $\sum_{i=1}^r m'_i = N'$ gives $\n' \succcurlyeq \m'$, so $w(\n', \m') \in \CF(\n').$  The map is clearly injective, so the claimed inequality holds.
\end{proof}

The special case where $N = N'$ is worth addressing separately, as we can precisely identify the difference between the sets $\CF(\n)$ and $\CF(\n')$.  Note that the corresponding statement also holds for $\LT(\n)$ thanks to Theorem \ref{theorem:CFLT}.
\begin{proposition}
\label{proposition:CFnCFn'}
Suppose that $\n = (n_1, \dots, n_i, n_{i+1}, \dots, n_r)$, and define $\n' := (n_1, \dots, n_i + 1, n_{i+1} -1, \dots, n_r)$.  Then
\begin{equation*}
|\CF(\n')|- |\CF(\n)| = |\CF(n_1, \dots, n_{i-1}, n_i + 1)| \cdot |\CF(n_{i+1} - 1, n_{i+2}, \dots, n_r)|.
\end{equation*}
\end{proposition}
\begin{proof}
This is a standard type of result that arises in the study of Catalan-type structures.  Recall the injective procedure from the proof of Lemma \ref{lemma:CFnCFn'}, and consider a word $w' = (\n',\m') \in \CF(\n')$ that is not the image of a word in $\CF(\n)$.  This means that $w = (\n, \m')$ is not a Catalan word, which can only happen if $n_1 + \dots + n_i = m'_1 + \dots + m'_i - 1$ (every other truncation of $\n$ and $\n'$ have the same sum).  Thus $w' = w_1 w_2$ with $w_1 \in \CF(n_1, \dots, n_{i-1}, n_i + 1)$ and $w_2 \in \CF(n_{i+1} - 1, \dots, n_r).$
\end{proof}
\begin{remark*}
For any pair $\n' \succcurlyeq \n$ with $N = N'$, there is a finite sequence of the adjacent shifts from Proposition \ref{proposition:CFnCFn'} that transforms $\n$ to $\n'$.  Therefore the difference $|\CF(\n')|- |\CF(\n)|$ can also be written explicitly as a finite sum of similar terms.
\end{remark*}

If the inputs are weakly increasing, then the correspondence in Corollary \ref{corollary:phiLT} combines with Proposition \ref{proposition:CFnCFn'} to imply an inequality for $\phi^*$ as well.
\begin{corollary}
\label{corollary:phi*ineq}
If $n'_1, \dots, n'_r$ is a weakly increasing sequence that satisfies $n'_i \leq n'_{i+1} - 1$ and $n'_{j-1} \leq n'_j - 1$ for some $i < j$, then
\begin{equation*}
\phi^*(n'_1, \dots, n'_i, \dots, n'_j, \dots, n'_r) \leq \phi^*(n'_1, \dots, n'_i + 1, \dots, n'_j - 1, \dots, n'_r).
\end{equation*}
\end{corollary}
\begin{proof}
Corollary \ref{corollary:phiLT} implies that $\phi^*(n'_1, \dots, n'r) = |\LT(n'_1, \dots, n'_r)|,$ and Proposition \ref{proposition:CFnCFn'} and Theorem \ref{theorem:CFLT} further give that
\begin{align}
\label{eq:LTineq}
|\LT(n'_1, \dots, n'_r)| & \leq |\LT(n'_1, \dots, n_{j-1} + 1, n_j - 1, \dots, n_r)| \\
& \leq |\LT(n'_1, \dots, n_{j-2} + 1, n_{j-1}, n_j - 1, \dots, n_r)| \notag \\
& \hspace{100 pt} \vdots \notag \\
& \leq |\LT(n'_1, \dots, n_i + 1, \dots, n_j - 1, \dots, n_r)| \notag \\
& = \phi^*(n'_1, \dots, n_i + 1, \dots, n_j - 1, \dots, n_r), \notag
\end{align}
where the last line again uses Corollary \ref{corollary:phiLT} and the inequality conditions for the $n_i$.
\end{proof}
\begin{remark*}
We only applied Corollary \ref{corollary:phiLT} to the first and last lines of (\ref{eq:LTineq}), and indeed, the Catalan words in the intermediate lines do not necessarily correspond bijectively to noncrossing pairings.  Thus it remains an open problem to prove Corollary \ref{corollary:phi*ineq} using only the combinatorics of noncrossing pairings.
\end{remark*}

\subsection{The Proof of Theorem \ref{Main Theorem}} \label{section Proof}  We now combine the machinery of Catalan structures from the previous subsections with the rotational symmetries of $\phi^*$ from Section \ref{section:Basic}.

We begin with additional definitions and notation.  Let $\Cr$ denote the cyclic group
of order $r$ (represented canonically by $[1, r]$), and consider arithmetic functions
$f:  \mathcal{C}_k \rightarrow \Te$.  For any $S \subseteq \Z_k$, the {\em weight} of $S$ is $f(S) : = \sum_{j \in S} f(j)$.  We are particularly interested in {\em intervals}, which are defined to be $\I(i,l):=\{i,i+1,...,i+l-1\} \subseteq \Cr$, where $\ell \in [1,r]$ is the {\em length}.
\begin{definition}
If $f, g: \Cr \rightarrow \Te$, we say $f$ is {\em cyclically dominated} by $g$ on $\I(i,l)$ if the sequence $(f(i), \ldots, f(i+\ell-1))$ is dominated by the sequence $(g(i), \ldots, g(i+\ell-1))$.  In this case, we write $g \succcurlyeq f.$
\end{definition}

We record some simple facts about domination and cyclic sequences that are fairly standard in the study of Catalan structures.
\begin{lemma}
\label{lemma:h0}
If $h(\Cr) = 0$, then $h \succcurlyeq 0$ on an interval $\I(i,r)$ for some $1 \leq i \leq r$.
\end{lemma}
\begin{proof}
Define the partial sums $H(i) := \sum_{j=1}^{i-1} h(j)$, and pick $i_0 \geq 1$ so that $H(i_0) \leq H(i)$ for all $i \geq 1$.  Then
\begin{equation*}
 h(\I(i_0,\ell)) = H(i_0+l) - H(i_0) \geq 0
\end{equation*}
for all $\ell$, so $h \succcurlyeq 0$ on $\I(i_0,r)$.
\end{proof}

\begin{corollary}
\label{corollary:fgdom}
If $f(\Cr) \leq g(\Cr)$, then $g \succcurlyeq f$ on some interval $\I(i,k)$.
\end{corollary}
\begin{proof}
Let $c = g(\Cr) - f(\Cr) \geq 0$.  Let $h = (g-f) - c$ so that $h(\Cr) = 0$. Lemma \ref{lemma:h0} implies that $h \succcurlyeq 0$ on some interval $\I(i,r)$, so $g \succcurlyeq f + c \succcurlyeq f$ as well.
\end{proof}

We now use these results to prove a much more general version of Corollary \ref{corollary:phi*ineq}.  The key to the arguments is that while $\phi(\n, \m)$ is rotationally invariant, $\CF(\n)$ is not.  We exploit this by first using cyclic dominance to appropriately rotate the noncrossing partitions, injecting into Catalan words, where (noncyclic) dominance then gives our desired inequalities.
\begin{theorem}
\label{theorem:phiCF}
If $w(\n, \m)$ is balanced, then for any $\n' \succcurlyeq \n,$ there is an $i$ such that
\[\phi(\n,\m) \leq |\CF(\Rot_i(\n'))|.\]
\end{theorem}
\begin{proof}
Corollary \ref{corollary:fgdom} implies that there is some $i$ such that $\Rot_i(\n') \succcurlyeq \Rot_i(\n),$ and by Lemma \ref{corollary:phisymmetry} we have $\phi(\n,\m) = \phi(\Rot_i(\n),\Rot_i(\m))$.  Corollary \ref{corollary:phiLT} combined with Theorem \ref{theorem:CFLT} and Lemma \ref{lemma:CFnCFn'} then give
\begin{equation*}
\phi(\n, \m) = \phi(\Rot_i(\n),\Rot_i(\m)) \leq |\CF(\Rot_i(n))| \leq |\CF(\Rot_i(\n'))|
\end{equation*}
as claimed.
\end{proof}

\begin{corollary}
\label{corollary:phiCnr}
If $r(n-1) < \sum_i n_i \leq n r$ then $\phi(\n,\m) \leq C^{(n)}_r.$
\end{corollary}
\begin{proof}
Define the length $r$ word $\n' := (n, n, \dots, n).$  Proposition \ref{prop Fuss-Catalan}, Corollary \ref{corollary:phiLT}, and Theorems \ref{theorem:CFLT} and \ref{theorem:phiCF} together imply that
$\phi(\n,\m) \leq |\CF(\n')| = \phi^*(\n') = C^{(n)}_r.$
\end{proof}
This is nothing more than a slightly rewritten version of Theorem \ref{Main Theorem}, and the proof is complete.

\section{Conclusion}
\label{section:Conclusion}

One of the chief difficulties in evaluating $\phi$ explicitly is that the rotational invariance of Corollary \ref{corollary:phisymmetry} is ``too strong''.  The fact that the arguments to $\phi$ may be freely rotated leads to very complicated formulas like Theorem \ref{theorem:phi*formula}, which is already greatly restricted to the symmetric case.  This is in contrast to the general formula for $\LT(\n)$ in Proposition \ref{proposition:LTformula}, which is relatively simple for all inputs.

In fact, in a very real sense, noncrossing pairings are the ``wrong'' generalization of Fuss-Catalan numbers.  In a forthcoming paper \cite{MS}, two of the authors consider a large number of other generalized Catalan structures (including the labeled trees and Catalan words of section {\ref{section:MainProof}, and also decompositions of polygons and lattice paths) that share many more properties with classical Catalan structures \cite{Stanley Catalan}.  Furthermore, we prove a very elegant symmetry result that can be stated as follows: If an $(n_1 + n_2 + \dots + n_r + 2)$-gon is decomposed into an $(n_1+2)$-gon, $(n_2+2)$-gon, \dots, and an $(n_r+2)$-gon, then the total number of such decompositions depends only on the sum $n_1 + n_2 + \dots + n_r$.

\end{document}